\documentclass{article}
\usepackage{amsmath}
\usepackage{amsthm}
\usepackage{amsfonts}
\usepackage{float}
\usepackage[dvips,final]{graphicx}
\usepackage{dcolumn}
\usepackage{comment} 
\usepackage{xspace}
\usepackage{url}
\usepackage{upgreek}
\usepackage{multicol}
\usepackage{pict2e,multirow}
\usepackage{verbatim}
\usepackage{wrapfig}
\usepackage{url}
\usepackage{tikz}
\usepackage[margin=0.35in]{caption}
\usepackage[colorlinks=true,citecolor=blue,urlcolor=blue]{hyperref}
\usepackage{subfigure}
\usepackage[T1]{fontenc}
\usepackage{titling}
\usepackage[hmargin=.720in]{geometry}

\newcommand{\sspcoef}{{\cal{C}}}
\newcommand{\dx}{\Delta x}
\newcommand{\dt}{\Delta t}
\newcommand{\Dt}{\Delta t}

\newcommand{\DtFE}{\dt_{\textup{FE}}}

\newcommand{\ceff}{\sspcoef_{\textup{eff}}}

\newcommand\vy{{\bf y}}
\newcommand\ve{{\bf e}}

\newcommand{\aij}{\alpha_{i,j}}
\newcommand{\bij}{\beta_{i,j}}

\newtheorem{rmk}{Remark}
\newtheorem{thm}{Theorem}
\newtheorem{lem}{Lemma}
\newtheorem{defn}{Definition}

\title{A Strong Stability Preserving  Analysis for Explicit Multistage Two-Derivative Time-Stepping Schemes
Based on Taylor Series Conditions.}
\author{ Zachary Grant$^1$\thanks{Corresponding author: grantzj@ornl.gov}, 
Sigal Gottlieb$^2$,
David C. Seal$^3$\\
{{\small $^1$Department of Mathematics, Department of Computational and Applied Mathematics,
Oak Ridge National Laboratory} }\\
{{\small $^2$Department of Mathematics, University of Massachusetts, Dartmouth} }\\
{{\small $^3$Department of Mathematics, U.S. Naval Academy.}} \normalsize
}
\date{}
\begin{document}

\maketitle
\bibliographystyle{siam}

\vspace{-0.5in}
\abstract{High order strong stability preserving (SSP)  time discretizations are often needed to ensure the nonlinear (and sometimes non-inner-product)
strong stability properties of   spatial discretizations specially designed for the solution of hyperbolic PDEs. Multiderivative time-stepping 
methods have recently been increasingly used for evolving hyperbolic PDEs, and the strong stability properties of these methods
are of interest. In our prior work we explored time discretizations that preserve the strong stability properties 
of spatial discretizations coupled with  forward Euler and a second derivative formulation. However, many spatial discretizations do not
satisfy strong stability properties when coupled with this second derivative formulation, but rather with a more natural Taylor series 
formulation. In this work we demonstrate  sufficient 
conditions for an explicit two-derivative multistage method to preserve the strong stability properties of spatial discretizations in a forward Euler
and Taylor series formulation. We call these strong stability preserving  Taylor series (SSP-TS) methods.
We also prove that the maximal order of SSP-TS methods is $p=6$, and define an optimization procedure
that allows us to find such SSP methods. Several types of these methods are presented and their efficiency compared. Finally,
these methods are tested on several  PDEs to demonstrate the benefit of  SSP-TS methods, the need for the SSP property, and the
sharpness of the SSP time-step in many cases.
}


 \normalsize
 
\section{Introduction}

The solution to a hyperbolic conservation law
\begin{eqnarray}
\label{pde}
	U_t + f(U)_x = 0,
\end{eqnarray}
may develop sharp gradients or discontinuities, which results in significant challenges  to the numerical
simulation of such problems. 
To ensure that it can handle the presence of a discontinuity, a spatial discretization is carefully
designed to satisfy some nonlinear stability properties, often in a non-inner-product sense, 
e.g. total variation diminishing, maximum norm preserving,  or positivity preserving properties. 
The development of high order spatial discretizations
that can handle discontinuities is a major research area \cite{cockburn1989,harten1983,kurganov2000,
liu1994, osher1984,sweby1984,tadmor1998,ToroADER,ToroADERWENO, ToroADERCL,ZhangShu2010}.

When the partial differential equation \eqref{pde} is semi-discretized we obtain the 
ordinary differential equation (ODE) 
\begin{eqnarray}
\label{ode}
u_t = F(u),
\end{eqnarray}
(where $u$ is a vector of approximations to $U$).
This formulation is often referred to as a method of lines (MOL) formulation, and has the advantage 
of decoupling the spatial and time-discretizations.
The  spatial discretizations designed to handle discontinuities 
ensure that when the semi-discretized equation \eqref{ode} 
is evolved using a forward Euler method
\begin{equation} \label{eqn:FEmethod}
	u^{n+1} = u^n + \dt F(u^{n}), 
\end{equation}
(where $u^n$ is a discrete approximation
to $U$ at time $t^n$)
 the numerical solution satisfies  the desired strong stability property,
such as total variation stability or positivity.
If the desired nonlinear  stability property
such as a norm, semi-norm, or convex functional,  is represented by $\| \cdot \|$,
the spatial discretization satisfies the monotonicity property
\begin{equation} \label{eqn:FEmonotonicity}
	\| u^n + \dt F(u^{n})  \| \leq \| u^n \|, 
\end{equation}
under the time-step restriction
\begin{equation} \label{eqn:FEcond}
	 0  \leq \dt \leq \Delta t_{FE}.
\end{equation}

In practice, in place of the first order time discretization \eqref{eqn:FEmethod}, we typically require a 
higher-order time integrator, that preserves the  strong stability  property 
\begin{equation} \label{eqn:monotonicity}
	\| u^{n+1} \| \le \|u^n\| ,
\end{equation}	
perhaps under a modified time-step restriction. 
For this purpose, time-discretizations with good linear stability properties 
or even with nonlinear inner-product stability properties are not sufficient.
Strong stability preserving (SSP) time-discretizations  were developed
to address this need. SSP multistep and Runge--Kutta  methods satisfy the strong stability property \eqref{eqn:monotonicity} for any function $F$, any initial condition, 
and any convex functional $\| \cdot \|$ under some time-step restriction,  provided only that \eqref{eqn:FEmonotonicity} is satisfied.

Recently, there has been interest in exploring the SSP properties of multiderivative Runge--Kutta methods,
also known as multistage multiderivative methods.
Multiderivative Runge--Kutta (MDRK) methods were first considered in 
\cite{obreschkoff1940,Tu50,StSt63,shintani1971,shintani1972,KaWa72,KaWa72-RK,
mitsui1982,ono2004, tsai2010},  and later 
explored for use with  partial differential equations (PDEs) \cite{sealMSMD2014,tsai2014, LiDu2016a,LiDu2016b,LiDu2018}.
These methods have a form similar to Runge--Kutta methods but use an additional derivative 
$  \dot{F} = u_{tt} \approx U_{tt} $ to allow for higher order.
The  SSP properties of these methods were discussed  in \cite{MSMD, Nguyen-Ba2010}.
In \cite{MSMD}, a method is defined as SSP-SD if it satisfies the strong stability property \eqref{eqn:monotonicity} 
for any function $F$, any initial condition,  and any convex functional $\| \cdot \|$ under some time-step restriction,  
provided  that \eqref{eqn:FEmonotonicity} is satisfied for any $ \Delta t \leq  \Delta t_{FE}$  and an additional condition of the form 
\[ \|u^n +   \Delta t^2 \tilde{F}(u^n) \| \leq \| u^n  \| \] is satisfied for any 
$ \Delta t \leq \tilde{K} \Delta t_{FE}$. These conditions allow us to find a wide variety of  time-discretizations, called SSP-SD time-stepping methods,
but they limit the type of spatial discretization that can be used in this context.

In this paper, we present a different approach to the SSP analysis, which is more along the lines 
of the idea in \cite{Nguyen-Ba2010}. For this analysis, we use, as before, the forward Euler base 
condition \eqref{eqn:FEmonotonicity}, but add to it a Taylor series condition of the form  
\[ \|u^n + \dt F(u^{n})  +   \frac{1}{2} \Delta t^2 \tilde{F}(u^n) \| \leq \| u^n  \|  \] 
that holds for any   $ \Delta t \leq K \Delta t_{FE}$.
Compared to those studied in \cite{MSMD}, this pair of base conditions allows for more flexibility in the choice of spatial
discretizations  (such as the methods that satisfy a Taylor series condition in  \cite{DumbserBalsara1,QiuDumbserShu,DaruTenaud,sealMSMD2014}), 
at the cost of more limited variety of time discretizations. 
We call the methods that preserve the strong stability properties of these base conditions strong stability preserving Taylor series  
(SSP-TS) methods.
The goal of this paper is to study a class of methods that is suitable for
use with existing spatial discretizations, and present families of such SSP-TS
methods that are optimized for the relationship between the forward-Euler time-step
$\DtFE$ and the Taylor series timestep $K \Delta t_{FE}$.

In the following subsections we describe SSP Runge--Kutta time discreizations and 
present explicit multistage two-derivative methods. We then motivate the need for methods that preserve the nonlinear stability properties of
  the forward Euler and Taylor series base conditions.
In Section  2 we formulate the SSP optimization problem for finding  explicit two-derivative methods which can be written as the convex combination
of forward Euler and Taylor series steps with the  largest allowable time-step, which we will later use to find optimized methods. In Subsection
2.1 we explore the relationship between SSP-SD methods and SSP-TS methods. 
In Subsection 2.2 we prove that there are order barriers associated with explicit  two-derivative methods that preserve the properties
of forward Euler and Taylor series steps with a positive time-step. In Section 3 we present the SSP coefficients of the optimized methods we obtain.
The methods themselves can be downloaded from our github repository \cite{SSPTSgithub}. 
In Section 4 we demonstrate how these methods perform on specially selected test cases, and in Section 5 we present our conclusions.

\subsection{SSP methods} \label{SSPmethods}

It is well-known  \cite{shu1988b, SSPbook2011} that some multistep and Runge--Kutta methods can be decomposed into convex combinations
of forward Euler steps, so that any  convex functional  property satisfied by \eqref{eqn:FEmonotonicity} will be {\em preserved}
by these higher-order time discretizations. If we re-write the $s$-stage explicit Runge--Kutta method in the Shu-Osher
form  \cite{shu1988},
\begin{eqnarray}
\label{rkSO}
y^{(0)} & =  & u^n, \nonumber \\
y^{(i)} & = & \sum_{j=0}^{i-1} \left( \aij y^{(j)} +
\dt \bij F(y^{(j)}) \right), \; \; \; \; i=1, . . ., s\\
 u^{n+1} & = & y^{(s)}  \nonumber
\end{eqnarray}
it is clear that if all the coefficients $\aij$ and $\bij$ are non-negative, and provided $\aij$ is zero only if 
its corresponding $\bij$ is zero, then each stage can be written as a convex combination of forward Euler 
steps of the form \eqref{eqn:FEmethod}, and be bounded by:
 \begin{eqnarray*}
\| y^{(i)}\| & =  & 
\left\| \sum_{j=0}^{i-1} \left( \aij y^{(j)} + \dt \bij F(y^{(j)}) \right) \right\|   \\
& \leq &  \sum_{j=0}^{i-1} \aij  \, \left\| y^{(j)} + \dt \frac{\bij}{\aij} F(y^{(j}) \right\|  
 \leq  \sum_{j=0}^{i-1} \aij  \, \left\| y^{(j)}   \right\|  \\
\end{eqnarray*}
under the condition  $ \frac{\bij}{\aij}  \dt \leq \DtFE$. By the  consistency condition $\sum_{j=0}^{i-1} \aij =1$,
we now  have $ \| u^{n+1}\| \leq \| u^{n}\| $, under the condition  
\begin{eqnarray}
	\label{rkSSP}
	\dt \leq \sspcoef  \DtFE \; \; \; \;  \mbox{where} \; \;  \sspcoef = \min_{i,j}  \frac{\aij}{\bij}
\end{eqnarray}
where if  any of the $\beta$s  are equal to zero, the corresponding  ratios are considered infinite.

 If a method can be decomposed into such a convex combination of \eqref{eqn:FEmethod}, with a positive value of 
$\sspcoef>0$ then the method is called {\em strong stability preserving} (SSP), and the value $\sspcoef$ is called
the  {\em SSP coefficient}. 
SSP methods  guarantee the  strong stability properties of any  spatial discretization, 
provided {\em only} that these properties are satisfied when using the forward Euler method.
The convex combination approach guarantees that 
the intermediate stages in a Runge--Kutta method  satisfy the desired strong stability property as well.
The convex combination approach clearly provides a sufficient condition for preservation of strong stability.
Moreover, it has also be shown that this condition is necessary
 \cite{ferracina2004, ferracina2005,higueras2004a, higueras2005a}.
 
 Second and third order explicit Runge--Kutta methods \cite{shu1988} 
 and later fourth order methods \cite{SpiteriRuuth2002, ketcheson2008} were found that 
 admit such a convex combination decomposition with $\sspcoef>0$.
 However, it has been proven that explicit  Runge--Kutta methods with positive SSP coefficient cannot be more than
fourth-order accurate \cite{kraaijevanger1991,ruuth2001}.
 
The  time-step restriction \eqref{rkSSP} is comprised of two
distinct factors: (1) the term $\DtFE$ that is a property of the spatial discretization, and 
(2) the SSP coefficient $\sspcoef$ that is a property of the time-discretization. 
Research on  SSP  time-stepping methods for hyperbolic PDEs has primarily focused on 
finding high-order time discretizations with the largest allowable time-step  $\Dt \le \sspcoef \DtFE$ by maximizing
the   {\em SSP coefficient} $\sspcoef$ of the method.

High order methods can also be obtained by adding 
more steps (e.g. linear multistep methods)   or more derivatives (Taylor series methods).
Multistep  methods that are SSP have been found \cite{SSPbook2011}, and 
explicit multistep SSP methods exist of very high order  $p>4$, but have 
severely restricted SSP coefficients \cite{SSPbook2011}. 
These approaches can be combined with Runge--Kutta methods to obtain methods with multiple steps, 
and stages. Explicit multistep multistage methods that are SSP and have order $p>4$ have been developed as well \cite{tsrk,msrk}.

\subsection{Explicit Multistage two-derivative methods} \label{MSMDoc}


Another way to obtain higher order methods is to use higher derivatives combines with 
the Runge--Kutta approach. An explicit multistage two-derivative time integrator is given by:
\begin{eqnarray}
\label{MSMD}
y^{(i)} & = &  u^n +  \dt \sum_{j=1}^{i-1}  a_{ij} F(y^{(j)}) +
\dt^2   \sum_{j=1}^{i-1}   \hat{a}_{ij} \dot{F}(y^{(j)}) , \; \; \; \; i=2, . . ., s \\
 u^{n+1} & = &  u^n +  \dt \sum_{j=1}^{s} b_{j} F(y^{(j)}) +
\dt^2  \sum_{j=1}^{i-1}   \hat{b}_{j} \dot{F}(y^{(j)}) , \nonumber
\end{eqnarray}
where  $y^{(1)}  =  u^n $.

The coefficients can be put  into matrix vector form, where
\[ A = \left( \begin{array}{llll}
 0 & 0 &  \vdots & 0 \\
 a_{21} & 0 & \vdots & 0  \\
 \vdots & \vdots & \vdots & \vdots \\
 a_{s1} & a_{s2} & \vdots & 0\\
 \end{array} \right) , 
 \; \; 
 \hat{A} = \left( \begin{array}{llll}
 0 & 0 &  \vdots & 0 \\
 \hat{a}_{21} & 0 & \vdots & 0  \\
 \vdots & \vdots & \vdots & \vdots \\
 \hat{a}_{s1} & \hat{a}_{s2} & \vdots & 0 \\
 \end{array} \right) , \; \; 
b = \left(  \begin{array}{llll}  b_1 \vspace{0.1in} \\ b_2  \vspace{0.1in}
 \\ \vdots  \\ b_s \vspace{0.1in} \\ \end{array} \right) , \; \;
\hat{b} =\left(\begin{array}{llll}  \hat{b}_1 \vspace{0.1in} \\  \hat{b}_2  \vspace{0.1in}\\ \vdots \\  \hat{b}_s \vspace{0.1in} \\ \end{array} \right) .
  \] 
  We also define the vectors  $c = A \ve$ and $\hat{c} = \hat{A} \ve$, where $\ve$ is a vector of ones.

As in our prior work \cite{MSMD}, we focus on using explicit multistage two-derivative methods as time integrators for evolving
hyperbolic PDEs. For our purposes, the operator  $F$  is obtained by a spatial discretization of the term $U_t= -f(U)_x$ to obtain the 
system $u_t = F(u)$. Instead of  computing the second derivative term $\dot{F}$  directly from the definition of the spatial discretization $F$,
we approximate $\tilde{F} \approx \dot{F}$ by employing the Cauchy-Kowalevskaya procedure which
 uses the  PDE \eqref{pde} to replace the time derivatives by the spatial derivatives, and discretize these in space.

If the term $F(u)$ is   computed  using a conservative spatial  discretization $D_x$ applied to the flux: 
\begin{equation}
\label{eqn:F}
	F(u) = - D_x(f(u)),
\end{equation}
then we approximate the second derivative
\begin{equation}
\label{eqn:Ft}
	\dot{F}(u) = u_{tt} \approx U_{tt} = -f(U)_{xt} = - \left(f(U)_t \right)_x =   
- \left(f'(U) U_t \right)_x \approx 
- \tilde{D}_x \left( f'(u) u_t \right) = \tilde{F}(u),
\end{equation}
where a (potentially different) spatial differentiation operator $\tilde{D}_x$ is used.
Although these two approaches are different, the differences between them 
are of  high order in space, so that in practice, as long as the 
spatial errors are smaller than the temporal errors, we 
see the correct order of accuracy in time, as shown in \cite{MSMD}.

\subsection{Motivation for the new base conditions for SSP analysis}
In \cite{MSMD} we considered explicit multistage two-derivative methods and developed sufficient conditions for a type of
strong stability preservation  for these methods. We showed that explicit SSP-SD methods within this class can break this well-known order barrier for 
explicit Runge--Kutta methods.  In that work we considered two-derivative methods that preserve the strong stability property
satisfied by a function $F$ under a convex functional $\| \cdot \|$, provided that the conditions:
\begin{eqnarray} \label{FE}
	\mbox{\bf Forward Euler condition:} \;\;  \|u^n +\Delta t F(u^n) \| \leq \| u^n  \| 
	 \;\;   \mbox{for}   \;\;  \Delta t \leq \Delta t_{FE},
\end{eqnarray}
and 
\begin{eqnarray}  \label{SD}
\mbox{\bf Second derivative condition:}  \;\;   
\|u^n +   \Delta t^2 \tilde{F}(u^n) \| \leq \| u^n  \|  \;\;  \mbox{for}   \;\;    \Delta t \leq \tilde{K} \Delta t_{FE},
\end{eqnarray}
where $\tilde{K}$ is a scaling factor that compares the stability condition of the second derivative term
to that of the forward Euler term.
While the forward Euler condition is characteristic of all SSP methods (and has been justified by 
the observation that it is the circle contractivity condition in \cite{ferracina2004}), the second derivative condition
was chosen over the Taylor series condition: 
\begin{eqnarray}  \label{TS}
\mbox{\bf Taylor series condition:} \;\;   
\|u^n +  \Delta t F(u^n) +  \frac{1}{2} \Delta t^2 \tilde{F}(u^n) \| \leq \| u^n  \|  \;\;   \mbox{for}  \;\;    \Delta t \leq K \Delta t_{FE},
\end{eqnarray}
because it is more general. If the forward Euler \eqref{FE} and second derivative \eqref{SD} conditions are both satisfied, 
then the Taylor series condition \eqref{TS} will   be satisfied as well. Thus,
a spatial discretization that satisfies \eqref{FE} and \eqref{SD} will also satisfy \eqref{TS},
so that the SSP-SD concept in \cite{MSMD} allows for the most general time-discretizations. 
Furthermore, some methods of interest and importance in the literature 
cannot be written using a Taylor series  decomposition, most notably the unique 
two-stage fourth order method\footnote{Note 
that here we use $\dot{F}$ to indicate that these methods are designed for the exact time derivative of $F$. 
However, in practice we use the approximation $\tilde{F}$ as explained above.}
\begin{subequations} \label{2s4p}
\begin{eqnarray} 
y^{(1)} & = & u^n + \frac{1}{2} \dt F(u^n) + \frac{1}{8} \dt^2 \dot{F}(u^n) \\
u^{n+1} & = & u^n + \dt F(u^n) + \frac{1}{6} \dt^2 \dot{F}(u^n) +  \frac{1}{3} \dt^2 \dot{F}(y^{(1)}) 
\end{eqnarray}
\end{subequations}
which appears commonly in the literature on this subject \cite{LiDu2016a, LiDu2016b, LiDu2018}.
For these reasons, it made sense to first consider the SSP-SD property which relies on the
 pair of base conditions  \eqref{FE} and \eqref{SD}.

However, as we will see in the  example below, there are spatial discretizations for which the 
second derivative condition \eqref{SD} is not satisfied but the 
forward Euler condition \eqref{FE} and the Taylor series condition  \eqref{TS} are both satisfied.  
In such cases, the SSP-SD methods derived in \cite{MSMD} may not preserve the desired strong stability properties. 
{\bf The existence of such spatial discretizations is the main motivation for the current work}, in which we
re-examine the strong stability properties of the explicit two-derivative multistage method \eqref{MSMD} using the base 
 conditions \eqref{FE} and \eqref{TS}. 
 Methods that preserve the strong stability properties of \eqref{FE} and \eqref{TS} are called herein SSP-TS methods.
 The SSP-TS approach increases our flexibility in the choice of spatial discretization over the SSP-SD approach. 
 Of course, this enhanced flexibility in the choice of spatial discretization  is expected to result in limitations on 
 the time-discretization (e.g. the two-stage fourth order  method is SSP-SD but not SSP-TS).
 
To illustrate the need for time-discretizations that preserve the strong stability properties of spatial discretizations that satisfy 
\eqref{FE} and \eqref{TS}, but not \eqref{SD}, consider the one-way wave equation 
\[ U_t = U_x  \]
(here $f(U) = U $) where $F$ is defined  by the first-order upwind method 
\begin{equation}
\label{eqn:low-order-a}
	F(u^n)_j := \frac{1}{\dx} \left(u^n_{j+1} - u^n_j \right) \approx U_x( x_j ).
\end{equation}
When solving the PDE, we compute the operator $\tilde{F}$ by simply applying the  differentiation operator twice  
(note that $f'(U) = 1$)
\begin{equation}
\label{eqn:low-order-b}
 \dot{F} \approx \tilde{F} := \frac{1}{\dx^2} \left( u^n_{j+2}- 2 u^n_{j+1} + u^n_{j} \right).
 \end{equation}

We note that when computed this way, the spatial discretization $F$ coupled with forward Euler satisfies the total variation diminishing 
(TVD) condition:
\begin{eqnarray}
\label{FEex}
\| u^n + \Delta t F(u^n)  \|_{TV} \leq \|u^n \|_{TV} \; \;  \; \;  \mbox{for} \; \; \; \; \Delta t \leq \Delta x, 
\end{eqnarray}
while the Taylor series term using $F$ and $\tilde{F}$ satisfies the TVD
\begin{eqnarray}
\label{TSex}
\| u^n + \Delta t F(u^n) + \frac{1}{2}  \Delta t^2 \tilde{F}(u^n)  \|_{TV} \leq \|u^n \|_{TV} \; \;  \; \;  \mbox{for} \; \; \; \; \Delta t \leq \Delta x. 
\end{eqnarray}
In other words, these spatial discretizations satisfy the conditions \eqref{FE} and \eqref{TS} with $K=1$, in the total variation semi-norm. However, 
\eqref{SD} is not satisfied, so the methods derived in \cite{MSMD} cannot be used.
Our goal in the current work is to develop time discretizations that will preserve the desired  strong stability properties (e.g. the total
variation diminishing property) when using   spatial discretizations such as the upwind approximation 
\eqref{eqn:low-order-b}   that satisfy \eqref{FE} and \eqref{TS} but  not  \eqref{SD}.
 
 \begin{rmk}
 This simple first order motivating example is chosen because these spatial discretizations are provably TVD
 and allow us to see clearly why the Taylor series base condition \eqref{TS} is needed. In practice, we use higher order
spatial discretizations  such as WENO that do not have a theoretical guarantee of TVD, but perform well in practice. 
Such methods are considered in Examples 2 and 4 in the numerical tests, and provide us with similar results. 
 \end{rmk}
  
In this work we develop explicit  two-derivative multistage SSP-TS methods of the form \eqref{MSMD} that 
preserve the convex functional properties  of forward Euler  and Taylor series terms.
When  the spatial discretizations $F$ and $\tilde{F}$ that satisfy \eqref{FE} and  \eqref{TS} are coupled with such a
  time-stepping method,  the strong stability condition
  \[\|u^{n+1} \| \leq \| u^n \| \]
  will be preserved, perhaps under a different time-step condition
  \begin{equation}
  \dt \leq \sspcoef_{TS} \DtFE .
  \end{equation}
  If a method can be decomposed in such a way, with $\sspcoef_{TS} > 0$ we say that it is SSP-TS.  
  In the next section, we define an optimization problem that will allow us to find SSP-TS methods of the form
   \eqref{MSMD} with the largest possible SSP coefficient $\sspcoef_{TS}$.

\section{SSP Explicit Two-Derivative Runge--Kutta Methods} \label{SSPopt}

We consider the system of ODEs 
\begin{eqnarray} 
u_t = F(u)
\end{eqnarray}
resulting from a semi-discretization of the  hyperbolic conservation law \eqref{pde}
such that  $F$ satisfies  the {\em forward Euler} (first derivative) condition \eqref{FE}
\begin{eqnarray*}
	\mbox{\bf Forward Euler condition:} \;\;  \|u^n +\Delta t F(u^n) \| \leq \| u^n  \| 
	 \;\;   \mbox{for}   \;\;  \Delta t \leq \Delta t_{FE},
\end{eqnarray*}
for the desired stability property indicated by the convex functional $\| \cdot \|$.  

The methods we are interested in also require an appropriate approximation 
to the  second derivative in time \[ \tilde{F}  \approx \dot{F}  = u_{tt}.\]  
We assume in this work that $F$ and $\tilde{F}$  satisfy an additional condition
of the form \eqref{TS}
\begin{eqnarray*} 
\mbox{\bf Taylor series condition:} \;\;   
\|u^n +  \Delta t F(u^n) +  \frac{1}{2} \Delta t^2 \tilde{F}(u^n) \| \leq \| u^n  \|  \;\;   \mbox{for}  \;\;    \Delta t \leq K \Delta t_{FE},
\end{eqnarray*}
in the same convex functional $\| \cdot \|$,
where $K$ is a scaling factor that compares the stability condition of the Taylor series term
to that of the forward Euler term.

We wish to show that given conditions  \eqref{FE} and \eqref{TS}, 
the multi-derivative  method \eqref{MSMD} satisfies the desired monotonicity condition
under a given time-step. This is easier if we re-write the method 
\eqref{MSMD} in an equivalent matrix-vector form 
\begin{eqnarray} \label{MSMDvector}
\vy = \ve u^n + \Delta t S F(\vy) +  \Delta t^2 \hat{S} \tilde{F}(\vy),
\end{eqnarray}
where $\vy = \left(y^{(1)}, y^{(2)},  . . ., y^{(s)}, u^{n+1}\right)^T$, 
\[ 
	S= \left[ \begin{array}{ll} A & \textbf{0} \\ b^T & 0 \end{array} \right] \; \; \; \; \;
	\mbox{and} \; \; \; \; \; \hat{S}= \left[ \begin{array}{ll} \hat{A} & \textbf{0} \\ \hat{b}^T & 0  \end{array}  \right]
\] and $\ve$ is a vector of ones. As in prior SSP work, all the coefficients in $S$ and $\hat{S}$ must be non-negative
(see Lemma \ref{poscoef}).

We can now easily establish sufficient conditions for an  explicit method of the form \eqref{MSMDvector} to be SSP:
\begin{thm} \label{thm1} 
Given spatial discretizations $F$ and $\tilde{F}$ that satisfy \eqref{FE} and \eqref{TS},
an explicit two-derivative multistage method  of the form \eqref{MSMDvector} preserves  the strong stability property 
$ \| u^{n+1} \| \leq \|u^n \|$  under the time-step restriction $\Delta t \leq r \DtFE$ if it satisfies the conditions
\begin{subequations} 
\begin{align} 
 \left( I +r S  +  \frac{2 r^2}{K^2} \left( 1 - K \right)  \hat{S} \right)^{-1} \ve  \geq  0   \label{SSPcondition1} \\
r  \left( I +r S  +  \frac{2 r^2}{K^2} \left( 1 - K \right)  \hat{S} \right)^{-1}      \left(S-  \frac{2 r}{K} \hat{S}  \right)
 \geq  0   \label{SSPcondition2} \\
\frac{2 r^2}{K^2}  \left( I +r S  +  \frac{2 r^2}{K^2} \left( 1 - K \right)  \hat{S} \right)^{-1}   \hat{S} 
\geq 0   \label{SSPcondition3}
 \end{align} 
\end{subequations}
for some $r>0$. In the above conditions, 
the inequalities are understood component-wise. 
\end{thm}
\begin{proof} We begin with the method 
\[ \vy = \ve u^n + \Delta t S F(\vy) +  \Delta t^2 \hat{S} \tilde{F}(\vy), \]
and 
add the terms $r S \vy$ and $2\hat{r}(\hat{r}-r) \hat{S} \vy$ to both sides to obtain the
{\em canonical Shu-Osher form} of an explicit two-derivative multistage method:
\begin{eqnarray*}
\left( I +r S  + 2 \hat{r} ( \hat{r}-r)  \hat{S} \right)  \vy &=& u^n \ve + r (S - 2 \hat{r} \hat{S} ) \left( \vy + \frac{\Delta t}{r}  F(\vy) \right) \\
 && + 2  \hat{r}^2 \hat{S}  \left( \vy   + \frac{\Delta t}{\hat{r}}  F(\vy) + \frac{\Delta t^2}{2 \hat{r}^2} \tilde{F}(\vy) \right),   \\
 \vy & = & R (\ve u^n) + P \left( \vy + \frac{\Delta t}{r}  F(\vy) \right) + Q  \left( \vy + \frac{\Delta t}{\hat{r}}  F(\vy) +  \frac{\Delta t^2}{2 \hat{r}^2} \tilde{F}(\vy) \right),
\end{eqnarray*}
where \[
	R  = \left( I +r S  + 2 \hat{r} ( \hat{r}-r)  \hat{S} \right)^{-1}, \; \; \; \; \; \;
	P  =  r R \left(S- 2 \hat{r} \hat{S}  \right),  \; \; \; \; \; \;
	Q =   2 \hat{r}^2 R  \hat{S}. 
\]
If the elements of $P$, $Q$, and $R \ve$ are all non-negative, and if $ \left(R + P + Q\right) \ve = \ve$, then $\vy$ is 
 a convex combination of strongly stable terms 
\[  \| \vy \| \leq  R \|\ve u^n\| + P \| \vy + \frac{\Delta t}{r}  F(\vy) \| + Q  \| \vy + \frac{\Delta t}{\hat{r}}  F(\vy) +   \frac{\Delta t^2}{2\hat{r}^2} \tilde{F}(\vy) \| ,\]
and so is also strongly stable  under the time-step restrictions $ \dt \leq r \DtFE$ and $ \dt \leq K \hat{r}  \DtFE$.
In such cases, the optimal time-step is given by the minimum of the two. In the cases we encounter here, this minimum occurs 
when these two values are set equal,  so we require  $r=  K \hat{r}$.
Conditions   \eqref{SSPcondition1}--\eqref{SSPcondition3} now ensure that $P \geq 0 $, $Q\geq 0$, and $R \ve \geq 0$  component-wise for 
$\hat{r} = \frac{r}{K}$, and so the  method  preserves the strong stability condition 
$\| u^{n+1} \| \leq \| u^n \|$ under the time-step restriction
$\Delta t \leq r \Delta t_{FE} $.  Note that if this fact holds for a given value of $r>0$ then it also holds for all smaller positive values.
\end{proof}

\begin{defn} \label{def1} 
A method that satisfies the conditions in Theorem \ref{thm1} for values $r \in (0, r_{max}]$ 
is called a  {\em Strong Stability Preserving Taylor Series} (SSP-TS) method with an associated SSP  coefficient
\[ \sspcoef_{TS} = r_{max}.\]
\end{defn}
\begin{rmk}
Theorem \ref{thm1}  gives us the conditions for the method \eqref{MSMDvector} to be SSP-TS for any time-step $\dt \leq \sspcoef_{TS} \DtFE$. 
We note, however, that while the corresponding conditions for Runge--Kutta methods have been shown to be necessary as
well as sufficient, for the multi-derivative methods we  only show that these conditions are sufficient.
This is a consequence of the fact that we define this notion of SSP based on the  conditions \eqref{FE} and \eqref{TS}, but if a spatial
discretization also satisfies a different condition (for example, \eqref{SD})
 many other methods of the form \eqref{MSMDvector} also give strong stability preserving results.
Notable among these is the two-derivative two-stage fourth order method \eqref{2s4p} which is SSP-SD but not SSP-TS.
This means that solutions of \eqref{2s4p} can be shown to  satisfy the strong stability 
property $ \| u^{n+1} \| \leq \|u^n \|$ for positive time-steps, for the appropriate spatial discretizations,
  even though the conditions in Theorem \ref{thm1} are not satisfied.
\end{rmk}

This result allows us to formulate the search for optimal SSP-TS methods as an optimization problem, as in 
 \cite{MSMD, SSPbook2011, ketcheson2008, ketcheson2009} 
\begin{center}
\noindent
{\bf  Find the coefficient matrices $S$ and $\hat{S}$ \\
that maximize the value of
$\sspcoef_{TS} = \max r$ \\ such that the relevant order conditions (summarized in Appendix \ref{MDRKOrderCondition}) \\
and the SSP conditions 
\eqref{SSPcondition1}-\eqref{SSPcondition3}
are all satisfied. }\\
\end{center}

However, before we present the optimal
methods  in Section \ref{optimalmethods},
we present the theoretical results on the allowable order of multistage multi-derivative SSP-TS methods.


\subsection{SSP results for explicit  two-derivative Runge--Kutta methods}
In this paper, we consider explicit SSP-TS two-derivative multistage methods that  can 
be decomposed into a convex combination of \eqref{FE} and \eqref{TS}, and thus
preserve their strong stability properties.  
In our previous work \cite{MSMD} we studied SSP-SD methods of the form \eqref{MSMD} that can be written as convex combinations of
  \eqref{FE} and \eqref{SD}.  
The following lemma explains the relationship between these two notions of strong stability.
 \begin{lem}
 \label{lemma2}
Any  explicit method of the form \eqref{MSMD} that can be written as a convex combination of the forward Euler formula \eqref{FE} and the Taylor series 
  formula \eqref{TS} can also be written as a convex combination of the forward Euler formula \eqref{FE} and the second derivative
  formula \eqref{SD}.
    \end{lem}
    \begin{proof}
    We can easily see that any Taylor series step can be rewritten as a convex combination of the forward 
    Euler formula \eqref{FE} and the second derivative formula \eqref{SD}:
 \[ u^n +  \dt F(u^n) + \frac{1}{2} \dt^2 \tilde{F}(u^n)  
 = \alpha \left( u^n + \frac{\dt}{\alpha} F(u^n)    \right) + (1-\alpha) \left(    u^n + \frac{1}{2 (1-\alpha)} \dt^2 \tilde{F}(u^n)  \right) , \]
for any $0 < \alpha < 1$.
Clearly then, if a method can be decomposed into a convex combination of  \eqref{FE} and \eqref{TS}, and in turn \eqref{TS}
can be decomposed into a convex combination of \eqref{FE} and \eqref{SD}, then the method itself can be written as a 
convex combination of \eqref{FE} and \eqref{SD}.
 \end{proof}
 
 This result  recognizes that the SSP-TS methods we study  in this paper are  a subset of the  SSP-SD
 methods in  \cite{MSMD}.
  This allows us to use results about SSP-SD methods
  when studying the properties of SSP-TS methods.
 
 The following lemma establishes the Shu-Osher form of an SSP-SD method of the form  \eqref{MSMD}.
 This form allows us to directly observe the convex combination 
 of steps of the form \eqref{FE} and \eqref{SD}, and thus easily  identify the SSP coefficient $\sspcoef_{SD}$. 
 \begin{lem}
If an explicit  method of the form  \eqref{MSMD} written in the {\em Shu-Osher form}
\begin{eqnarray}
\label{MD_Shu-Osher}
y^{(1)} & = & u^n  \nonumber \\
y^{(i)} & = &    \sum_{j=1}^{i-1}  \alpha_{ij}  y^{(j)} + \dt \beta_{ij} F(y^{(j)})  
+   \hat{\alpha}_{ij}  y^{(j)} + \dt^2 \hat{\beta}_{ij}  \tilde{F}(y^{(j)}) , \; \; \; \; i=2, . . ., s+1 \\
 u^{n+1} & = &    y^{(s+1)}  \nonumber
\end{eqnarray}
has the properties that
\begin{enumerate}
\item all the coefficients are non-negative,
\item  $ \beta_{ij}=0$ whenever $\alpha_{ij}=0 $ 
\item $ \hat{\beta}_{ij}=0 $  whenever $\hat{\alpha}_{ij}=0 $ 
\end{enumerate}
then this method preserves the strong stability properties of \eqref{FE} and \eqref{SD} (i.e. 
is SSP-SD) for $\dt \leq \sspcoef_{SD} \DtFE$
 with \[\sspcoef_{SD} = \min_{i,j}  \left\{ \frac{\aij}{\bij}, \tilde{K} \frac{\hat{\alpha}_{ij}}{\hat{\beta}_{ij}} \right\}\].
\end{lem}
\begin{proof}
For each stage we have 
\begin{eqnarray*}
\| y^{(i)} \| & = &  \left\|  \sum_{j=1}^{i-1}  \alpha_{ij}  y^{(j)} + \dt \beta_{ij} F(y^{(j)})  
+  \hat{\alpha}_{ij}  y^{(j)} + \dt^2 \hat{\beta}_{ij}  \tilde{F}(y^{(j)})  \right\| \\
& \leq &   \sum_{j=1}^{i-1}  \alpha_{ij}  \left\| y^{(j)} + \dt \frac{\beta_{ij}}{\alpha_{ij}} F(y^{(j)})   \right\|
+ \sum_{j=1}^{i-1}  \hat{\alpha}_{ij}  \left\| y^{(j)} + \dt^2 \frac{\hat{\beta}_{ij}}{\hat{\alpha}_{ij}}   \tilde{F}(y^{(j)}) \right\| \\
& \leq & \sum_{j=1}^{i-1}  \left( \alpha_{ij}  + \hat{\alpha}_{ij} \right) \|u^n\| 
\end{eqnarray*}
assuming only that for each one of these $\dt  \frac{\beta_{ij}}{\alpha_{ij}} \leq \DtFE$ and 
$\dt  \frac{\hat{\beta}_{ij}}{\hat{\alpha}_{ij}} \leq  \tilde{K} \DtFE$. The result immediately follows from the fact that
for each $i$ we have $  \sum_{j=1}^{i-1}  \left( \alpha_{ij}  + \hat{\alpha}_{ij} \right) =1$ for consistency.
\end{proof}
 
When a method is written in the block Butcher form \eqref{MSMDvector}, we can decompose it
into a  canonical Shu-Osher form. 
\begin{eqnarray*}
Y&=&(I+rS+\hat{r}\hat{S})^{-1}eu^n + (I+rS+\hat{r}\hat{S})^{-1}S(Y+\frac{\Delta t}{r} F(Y) )+ 
  (I+rS+\hat{r}\hat{S})^{-1}\hat{S}(Y + \frac{\Delta t^2}{\hat{r}} \tilde{F}(Y)).
\end{eqnarray*}
This allows us to define an  SSP-SD method directly from the Butcher coefficients.
  
  \begin{defn} \label{def2} 
Given spatial discretizations $F$ and $\tilde{F}$ that satisfy \eqref{FE} and \eqref{SD},
an explicit two-derivative multistage method  of the form \eqref{MSMDvector} is called 
a Strong Stability Preserving Second Derivative (SSP-SD) method
with and associated  SSP coefficient $\sspcoef_{SD} =   \min\{r_{max}, \tilde{K}\hat{r}_{max} \} $ if it satisfies the conditions
\begin{subequations} 
\begin{align} 
(I+rS+\hat{r}\hat{S})^{-1}e&\geq 0 \label{3}\\
 (I+rS+\hat{r}\hat{S})^{-1}rS& \geq 0 \label{4}\\
  (I+rS+\hat{r}\hat{S})^{-1}\hat{r}\hat{S}&\geq 0. \label{5}
 \end{align} 
\end{subequations}
for all $r = (0, r_{max}]$ and $\hat{r} = (0, \hat{r}_{max}]$. In the above conditions, 
the inequalities are understood component-wise. 
\end{defn}

 The relationship between the coefficients in \eqref{MSMD} and \eqref{MD_Shu-Osher} allows us to conclude that
 the matrices $S$ and $\hat{S}$ must contain only non-negative coefficients.
\begin{lem} \label{poscoef}
If an explicit  method of the form  \eqref{MSMD} can be converted to the Shu-Osher form
 \eqref{MD_Shu-Osher} with all non-negative coefficients $\alpha_{ij}, \beta_{ij}, \hat{\alpha}_{ij}, \hat{\beta}_{ij},$
 for all $i, j$, 
  then the  coefficients $a_{ij}, b_j, \hat{a}_{ij}, \hat{b}_j$ must be all non-negative as well.
\end{lem}
\begin{proof}
The transformation between \eqref{MSMD} and \eqref{MD_Shu-Osher} is given by
$ a_{21} = \beta_{21}$ and $ \hat{a}_{21} = \hat{\beta}_{21} $
and, recursively,
\begin{subequations}
\begin{eqnarray}
a_{ij} & = & \beta_{ij} + \sum_{k=j+1}^{i-1}  ( \alpha_{ik} + \hat{\alpha}_{ik} ) a_{kj} \label{aij} \\
\hat{a}_{ij} & = & \hat{\beta}_{ij} + \sum_{k=j+1}^{i-1}  ( \alpha_{ik} + \hat{\alpha}_{ik} ) \hat{a}_{kj} \label{hataij}  \\
b_{j} & = & \beta_{s+1,j} + \sum_{k=j+1}^{s}  ( \alpha_{s+1,k} + \hat{\alpha}_{s+1,k} ) a_{kj}  \label{bj} \\
\hat{b}_{j} & = & \hat{\beta}_{s+1,j} + \sum_{k=j+1}^{s}  ( \alpha_{s+1,k} + \hat{\alpha}_{s+1,k} ) \hat{a}_{kj}  \label{hatbj}
\end{eqnarray}
\end{subequations}

Clearly then, \[ \beta_{21} \geq 0   \implies  a_{21} \geq 0 \]
and
\[ \hat{\beta}_{21} \geq 0  \implies  \hat{a}_{21} \geq 0 .\]

From there we proceed recursively:  
Given that all $\alpha_{ij} \geq 0$ and $\beta_{ij} \geq 0$ for all $i,j$,  and that  $a_{kj} \geq 0$  and
$\hat{a}_{kj} \geq 0$ for all $1 \leq j < k \leq i-1$, then by the formulae \eqref{aij} and \eqref{hataij} we have
 $a_{ij} \geq 0$ and $\hat{a}_{ij} \geq 0$.
 
 Now given $\alpha_{ij} \geq 0$ and $\beta_{ij} \geq 0$ for all $i,j$ 
 and $a_{kj} \geq 0$  and $\hat{a}_{kj} \geq 0$ for all $1 \leq j < k \leq s$, the formulae \eqref{bj} and \eqref{hatbj} give the result
$b_{j} \geq 0$ and $\hat{b}_{i} \geq 0$.
Thus, all the coefficients  $a_{ij}, \hat{a}_{ij}, b_j, \hat{b}_j$ must be all non-negative.

\end{proof}

We wish to study only those methods  for which the  Butcher form \eqref{MSMD} is unique. 
To do so, we follow Higueras \cite{higueras2009} in extending the 
reducibility definition of  Dahlquist and Jeltsch \cite{Jeltsch2006}.
Other notions of reducibility exist, but for our purposes it is sufficient to define irreducibility as follows:

\begin{defn} A two-derivative multistage method of the form \eqref{MSMD} is {\bf DJ-reducible}
if there exist sets $T_1$ and $T_2$ such that  $T_1 \neq \emptyset$, $T_1 \bigcap T_2 = \emptyset$,
$T_1 \bigcup T_2 = [1,2, . . . , s]$, and 
\[b_j = \hat{b}_j = 0, \; \; \; \; j \in T_1\]
\[a_{ij} = \hat{a}_{ij} = 0, \; \; \; \; i \in T_1, \; j \in T_2.\]
We say a method is  {\bf irreducible} if it is not DJ-reducible.
\end{defn}

 \begin{lem} \label{lemma3}
   An irreducible  explicit SSP-SD method of the form \eqref{MSMDvector} must satisfy 
   the (componentwise) condition 
   \[b+\hat{b} > 0.\]
    \end{lem}
  \begin{proof} 
An SSP-SD method in the block form \eqref{MSMDvector}, must satisfy conditions \eqref{3}--\eqref{5} for
$0 < r \leq r_{max}$ and $0 < \hat{r} \leq \hat{r}_{max}$.
The non-negativity of  \eqref{4} and \eqref{5}  requires their sum to be non-negative as well,
\begin{eqnarray*}
(I+rS+\hat{r}\hat{S})^{-1}(rS+\hat{r}\hat{S})\geq 0. \label{PosSum}
\end{eqnarray*}
Note that these matrices commute, so we have
\begin{eqnarray*}
(rS+\hat{r}\hat{S}) (I+rS+\hat{r}\hat{S})^{-1} \geq 0. \label{PosSumB}
\end{eqnarray*}

Recalling the definition of the matrix $S$, we have
\begin{eqnarray*}
e_{s+1}(rS+\hat{r}\hat{S}) (I+rS+\hat{r}\hat{S})^{-1} \geq 0 \\
 ([ (rb+\hat{r}\hat{b}),0])(I+rS+\hat{r}\hat{S})^{-1}\geq 0 
\end{eqnarray*}
Now, we can expand the inverse as
\[  ([ (rb+\hat{r}\hat{b}),0]) (I-(rS+\hat{r}\hat{S})+ (rS+\hat{r}\hat{S})^2 - (rS+\hat{r}\hat{S})^3 + . . .  ) \geq 0  \]

Because the positivity must hold for arbitrarily small $r<<1$ and $\hat{r} <<1$ we can stop our expansion after the linear term,
and require 
\[  ([ (rb+\hat{r}\hat{b}),0]) (I-(rS+\hat{r}\hat{S})) \geq 0 , \]
which is 
\begin{eqnarray*}
  (rb+\hat{r}\hat{b})(I-(rA+\hat{r}\hat{A}))\geq0 \implies
    (rb+\hat{r}\hat{b}) \geq(rb+\hat{r}\hat{b})(rA+\hat{r}\hat{A})).
    \end{eqnarray*}
    
Now we are ready to address the proof. Assume that  $j=J$ is the largest  value for which we have $b_J = \hat{b}_J=0$, 
\[ 0 =         (r b+\hat{r}\hat{b})e_{J} \geq (rb+\hat{r}\hat{b})(rA+\hat{r}\hat{A}) e_{J} \geq 0 \]
Clearly, then we have
\[  r\sum_i (rb_i+\hat{r}\hat b_i) A_{iJ}  +\hat{r}\sum_i (rb_i+\hat{r}\hat b_i) \hat{A}_{iJ} =0      \]      
Since the method is explicit, the matrices $A$ and $\hat{A}$ are lower triangular (i.e $a_{iJ}=\hat{a}_{iJ} = 0$ for $i \leq J$), 
so this condition becomes
\begin{eqnarray}
    \label{eqn:proof}
    r\sum_{i=J+1}^s (rb_i+\hat{r}\hat b_i) A_{iJ}  +\hat{r}\sum_{i=J+1}^s (rb_i+\hat{r}\hat b_i) \hat{A}_{iJ} =0      
\end{eqnarray}   
By the assumption above, we have $(b_i+\hat{b}_i) > 0$ for $i > J $, and  $r>0, \hat{r}>0$. 
Clearly, then, for \eqref{eqn:proof}  to hold we must require 
\[ A_{iJ} = \hat{A}_{iJ}=0 \; \; \; \; \; \mbox{ for} \; \; \; \;  i > J,\]
which, together with $b_J=\hat{b}_J =0$, makes the method DJ-reducible. Thus we have a contradiction.
\end{proof}
\noindent We note that this same result, in the context of additive Runge--Kutta methods, is due to Higueras \cite{higueras2009}.

\subsection{Order barriers} \label{OrderBarriers}
Explicit SSP Runge--Kutta methods with $\sspcoef >0$ are known to have an order barrier of four, while the
implicit methods have a barrier of six \cite{SSPbook2011}. This  follows from the fact that the order $p$ of  irreducible methods with nonnegative coefficients
depends on the stage order $q$ such that 
 \[q \geq \left\lfloor \frac{p-1}{2} \right\rfloor.\]
For explicit Runge--Kutta methods the first stage is a forward Euler step, so $q=1$ and thus $p \leq 4$, whereas for 
implicit Runge--Kutta methods the first stage is at most of order two, so that $q=2$ and thus $p \leq 6$.

For two-derivative multistage SSP-TS methods, we find that similar results hold. A stage order of $q=2$ is possible for explicit  two-derivative methods (unlike explicit Runge--Kutta methods)
 because the first stage can be  second order,  i.e., a Taylor series method. However, since the first stage can be no 
 greater than second order we have a bound on the stage order $q \leq 2$, which results in an order barrier of 
 $p \leq 6$ for these methods.   In the following results we establish these order barriers. 
 
\begin{lem}  \label{lemma1}
Given an irreducible SSP-TS method of the form \eqref{MSMD},
if  $b_j=0$ then the corresponding $\hat{b}_j =0$.
\end{lem}
\begin{proof}
In any SSP-TS method  the appearance of a second derivative term $\tilde{F}$ can only happen as part of a Taylor series term.
This tells us that the $\tilde{F}$ must be accompanied by the corresponding $F$,  meaning that whenever 
we have a nonzero $\hat{a}_{ij}$ or $\hat{b_j}$ term, the corresponding ${a}_{ij}$ or ${b_j}$ term must be nonzero. 
\end{proof}
 
      \begin{lem} \label{cor}
  Any irreducible explicit SSP-TS method of the form \eqref{MSMD} must satisfy the (componentwise) condition 
   \[b > 0.\]
     \end{lem}
     \begin{proof}
     Any irreducible method \eqref{MSMD} that can be written as a convex combination of \eqref{FE} and \eqref{TS} can also be written as a 
     convex combination of \eqref{FE} and \eqref{SD}, according to Lemma \ref{lemma2}. Applying Lemma \ref{lemma3} we obtain the
     condition $b+\hat{b} > 0 $, componentwise. Now, Lemma \ref{lemma1} tells us that if any component ${b}_j =0$ then its corresponding $\hat{b}_j =0$,
     so that $b_j +  \hat{b}_j  > 0$ for each $j$ implies that $b_j >0$ for each $j$.
     \end{proof}

  \begin{thm} \label{thm2}
Any irreducible explicit  SSP-TS method of the form \eqref{MSMD}
with order $p\geq 5$  must satisfy the stage order $q=2$ condition
\begin{eqnarray} \label{SO2}
 \tau_2 =  A c + \hat{c}  - \frac{1}{2}  c^2  = \mathbf{0} 
 \end{eqnarray}
 where the term $c^2$ is a component-wise squaring.
  \end{thm}    
  \begin{proof}
A method of order $p \geq 5$ must satisfy the 17 order conditions presented in the Appendix \ref{MDRKOrderCondition}. Three of those necessary 
conditions are\footnote{In this work we use  $\odot$ to denote component-wise multiplication.} 
   \begin{subequations}
   \begin{align}
	&  b^Tc^4 + 4\hat{b}^Tc^3 =\frac{1}{5}  \label{p5.1} \\
 	&  b^T\left( c^2 \odot Ac\right)  + b^T\left( c^2 \odot \hat{c}\right)+\hat{b}^Tc^3+2\hat{b}^T\left( c \odot Ac\right)+2\hat{b}^T\left( c \odot \hat{c}\right)=\frac{1}{10} \label{p5.2}\\
	&  b^T\left( Ac \odot Ac \right) +2b^T\left( \hat{c} \odot Ac\right)+b^T\hat{c}^2+ 2\hat{b}^T\left(c \odot Ac\right)+2\hat{b}^T\left(c \odot \hat{c}\right)=\frac{1}{20} . \label{p5.3}
   \end{align}  
    \end{subequations}
From this, we find that the following linear combination of these equations gives
\[  \frac{1}{4} \eqref{p5.1} - \eqref{p5.2} +   \eqref{p5.3} = b^T \left( A c +\hat{c}  - \frac{1}{2} c^2 \right)^2 =  b^T \tau_2^2 = 0  \]
(once again, the squaring here is  component-wise). 
Given the strict component-wise positivity of the vector $b$  according to Lemma \ref{cor} and the 
non-negativity of $\tau_2^2$,   this condition becomes $\tau_2 =\mathbf{0}$.
  \end{proof}

    \begin{thm}
Any irreducible explicit  SSP-TS method of the form \eqref{MSMD}  cannot have order $p=7$. 
    \end{thm}  
    \begin{proof}
    This proof is similar to the proof of Theorem \ref{thm2}.
    The complete list of additional order conditions for seventh order is lengthy and beyond the scope of this work.
    However, only three of these conditions are needed for this proof. These are:
   \begin{subequations}
   \begin{align}
&b^T c^6 +6\hat{b}c^5 =\frac{1}{7}    \label{p7.1} \\
&b^T\left(Ac^2 \odot c^3\right) + 2b^T\left(\hat{A}c \odot c^3 \right) +\hat{b}^Tc^5 +3\hat{b}^T\left( Ac^2 \odot c^2\right) + 6\hat{b}^T\left(\hat{A}c \odot c^2 \right)= \frac{1}{21}       \label{p7.2} \\
&b^T\left( Ac^2 \odot Ac^2 \right) + 4b^t \left(\hat{A}c \odot Ac^2  \right) + 4b^T \left(\hat{A}c \odot \hat{A}c \right) + 4\hat{b}^T\left(\hat{A}c \odot c^2 \right) +2 \hat{b}^T \left( Ac^2 \odot c^2 \right)=\frac{1}{63}  .   \label{p7.3} 
   \end{align}
    \end{subequations}
    Combining these three equations we have:
      \[   \frac{1}{9}\eqref{p7.1}  -\frac{2}{3}\eqref{p7.2} + \eqref{p7.3} = b^T\left(Ac^2+\hat{A}c-\frac{c^3}{3}  \right)^2=0.\] 
From this we see that any seventh order method of the form \eqref{MSMD} which admits a decomposition of a convex combination
    of \eqref{FE} and \eqref{TS}, must satisfy the  stage order $q=3$ condition
    \[   \tau_3=\left(Ac^2+\hat{A}c-\frac{c^3}{3}\right)=\mathbf{0}.\] 
      
    However, as noted above,  the first stage of the explicit two-derivative multistage method \eqref{MSMD} has the form
    \[ u^n + a_{21} \dt F(u^n) + \hat{a}_{21} \dt^2 \tilde{F}(u^n) \] which can be at most of 
    second order. This means that the stage order of explicit two-derivative multistage methods can be at most $q=2$, and so the $\tau_3=0$ condition
    cannot be satisfied.
    Thus, the result of the theorem follows.
      \end{proof}
      
  Note that the order barriers do not hold for   SSP-SD methods, because SSP-SD methods do not 
     requirement that all components of the vector $b$ must be strictly positive.

\section{Optimized SSP Taylor Series Methods} \label{optimalmethods}

In Section \ref{SSPopt} we formulated the  search
 for optimal SSP two-derivative methods as: 
 \begin{center}
\noindent
{\bf  Find the coefficient matrices $S$ and $\hat{S}$ 
that maximize the value of
$\sspcoef_{TS} = \max r$ \\ such that the relevant order conditions 
and the SSP conditions 
\eqref{SSPcondition1}-\eqref{SSPcondition2}
are all satisfied. }\\
\end{center}
To accomplish this, we develop and  use  a   {\sc matlab} optimization code  \cite{SSPTSgithub}  (similar to Ketcheson's code \cite{ketchcodes})
for finding optimal two-derivative multistage methods that preserve the SSP properties \eqref{FE} and \eqref{TS}.
The  SSP coefficients of the optimized SSP explicit multistage two-derivative methods of order up to $p=6$ (for different values of $K$)
are presented in this section.

 We considered three types of methods:
 \begin{enumerate}
 \item[{\bf (M1)}] Methods that have the general form \eqref{MSMD} with no simplifications.
 \item[{\bf (M2)}]  Methods that are constrained to satisfy the  stage order two ($q=2$) requirement \eqref{SO2}:
  \[ \tau_2 = A c +\hat{c} - \frac{1}{2} c^2  = 0.\]
  \item[{\bf (M3)}] Methods that satisfy the stage order two ($q=2$) \eqref{SO2} requirement and require only $\dot{F}(u^n)$, 
so they have only one second derivative evaluation. This is equivalent to requiring that all values in $\hat{A}$ and $\hat{b}$,
 except those on the first column of the matrix and the first element of the vector, be zero.
\end{enumerate}
We refer to the methods by type, number of stages, order of accuracy, and value of $K$. For example,
an SSP-TS method of type (M1) with $s=5$  and $p=4$, optimized for the value of $K=1.5$ would be referred to as 
SSP-TS M1(5,4,1.5) or as SSP-TS M1(5,4,1.5). For comparison, we refer to methods from \cite{MSMD} that are SSP in the sense that they preserve the properties of  
the spatial discretization coupled with \eqref{FE} and \eqref{SD} as SSP-SD MDRK(s,p,K) methods.

\begin{figure}[h]
   \centering
          \includegraphics[width=.8\textwidth]{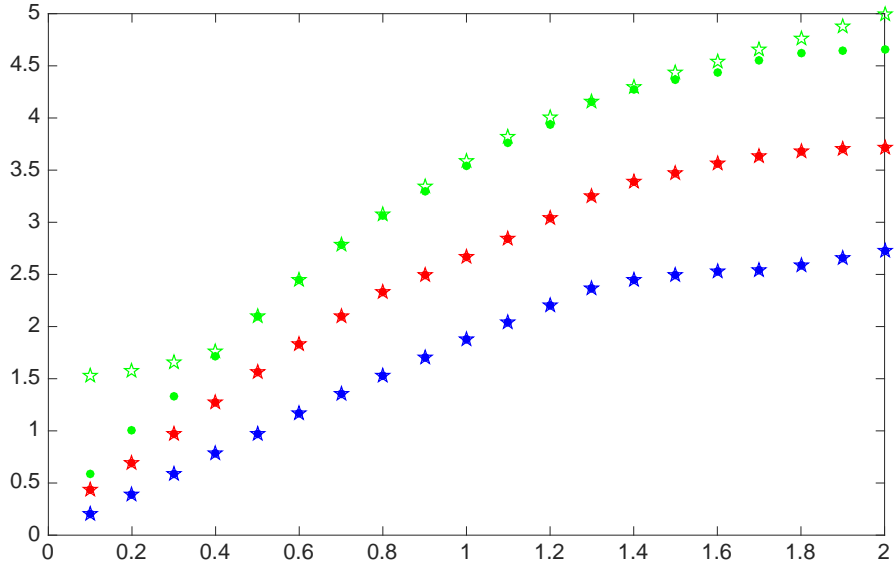}
\caption{The SSP-TS coefficient $\sspcoef_{TS}$ (on y-axis)
of fourth order SSP-TS M1 and M2 methods with  $s=\textcolor{blue}{3},\textcolor{red}{4},\textcolor{green}{5}$ stages
plotted against the value of $K$ (on x-axis). The open stars indicate methods of type (M1) 
while the filled circles are methods of type (M2). Filled stars are (M1) markers overlaid with (M2) markers indicating close if not equal SSP coefficients.
}
\label{Fig4thO} 
\end{figure}

 \subsection{Fourth Order Methods}
Using the optimization approach described above, we find fourth order methods  with $s=3,4,5$  
stages for a range of $K=0.1, . . . , 2.0$.
In Figure \ref{Fig4thO} we show the SSP coefficients of methods of SSP-TS methods of type (M1) and (M2) 
with $s=3,4,5$  (in blue, red, green) plotted against the value of $K$. The open stars indicate methods of type (M1) 
while the filled circles are methods of type (M2). Filled stars are (M1) markers overlaid with (M2) markers indicating close if not equal SSP coefficients.
  
  \smallskip
  
\noindent {\bf Three Stage Methods}:
Three stage SSP-TS methods with fourth order accuracy exist, and all these have stage order two ($q=2$), so they are
all of type (M2).  Figure \ref{Fig4thO} shows the SSP coefficients of these methods in blue. The (M3) methods have an SSP coefficient 
\[ \sspcoef_{TS} =  \left\{ \begin{array}{ll} 
\frac{2 K }{K+1} &  \mbox{for} \; \; \;  K \leq 1 \\
1 &  \mbox{for} \; \; \;  K \geq 1. \\
\end{array}  \right. \] 
For the case where $K\geq1$ we obtain the following optimal (M3) scheme with a SSP coefficient $\sspcoef_{TS} =1$:
\[
\begin{array}{l l}
\vspace{.25 cm}
y^{(1)} &= u^n\\ \vspace{.25 cm}
y^{(2)} &= u^n +\Delta t F(y^{(1)}) + \frac{1}{2} \Delta t^2 \dot{F}(y^{(1)})\\\vspace{.25 cm}
y^{(3)} &=u^n + \frac{1}{27}  \Delta t \left(14F(y^{(1)})  + 4F(y^{(2)})\right) +\frac{2}{27}\Delta t^2 \dot{F}(y^{(1)})\\\vspace{.25 cm}
u^{n+1}&= u^n + \frac{1}{48}  \Delta t \left(17 F(y^{(1)})  + 4F(y^{(2)}) +27F(y^{(3)})\right)+\frac{1}{24} \Delta t^2 \dot{F}(y^{(1)}).\\
\end{array}
\] 
When $K\leq1$ we have to modify the coefficients accordingly to obtain the maximal value of $\sspcoef_{TS}$ as defined above.  
Here we provide the non-zero coefficients for this family of  M3(3,4,K) as a function of K: 
\[
\begin{array}{l l l}
\vspace{.25 cm}
a_{21} = \frac{K+1}{2} & b_1 =  \frac{3K^5 - 9K^4 - 22K^3 + 30K^2 + 21K + 11}{ 3(K - 3)^2(K + 1)^3 } & \hat{a}_{21}=\frac{(K + 1)^2}{8}\\\vspace{.25 cm}
a_{31} = \frac{(K + 1)(- K^3 - 2K^2 + 14K + 3)}{2(K + 2)^3} & b_2 = \frac{2K}{3(K + 1)^3} & \hat{a}_{31}= \frac{K(- K^2 + 2K + 3)^2}{8(K + 2)^3} \\\vspace{.25 cm}
a_{32} = \frac{(K + 1)(K - 3)^2}{2(K+2)^3} &  b_3 =  \frac{2(K + 2)^3}{3(K - 3)^2(K + 1)^3}&  \hat{b}_1 =  -\frac{- 3K^3 + 3K^2 + K + 1}{6(K - 3)(K + 1)^2 }. \\
\end{array}
\]


 In Table \ref{M2M33s4p} we compare the SSP coefficient  of three stage fourth order SSP-TS 
 methods of type (M2) and (M3) for a selection of values of $K$. 
 Clearly, the  (M3) methods have a much smaller SSP coefficient than the (M2) methods. However, a better measure of efficiency is 
 the {\em effective SSP coefficient} computed by  normalizing for the number of function 
 evaluations required, which is $2s$ for the (M2) methods, and $s+1$ for the (M3) methods. If 
 we consider the  effective SSP coefficient, we find that while the (M2) methods are more
 efficient for the larger values of $K$, for smaller values of $K$ the (M3) methods are more efficient.

 \begin{table}
 \begin{center}
 \begin{tabular}{|l|lllllll|} \hline
 &K = &  0.1 & 0.2 & 0.5 & 1.0 & 1.5 & 2.0 \\ \hline
 (M2)  &  $\sspcoef_{TS}$  = & 0.1995  & 0.3953 & 0.9757 & 1.8789 & 2.4954 & 2.7321 \\ 
 &   $\ceff$= & 0.0333  &  0.0659   &  0.1626  &  0.3131  &  0.4159 &   0.4553 \\ \hline
  (M3) &  $\sspcoef_{TS}$  = & 0.1818  & 0.3333 &  0.6667  &  1.0000 & 1.0000 & 1.0000 \\ 
 &    $\ceff$ = & 0.0454   & 0.0833   &  0.1667  &  0.2500   & 0.2500 &   0.2500 \\ \hline
 \end{tabular}
 \end{center}
 \caption{SSP-TS  coefficients of three stage fourth order SSP-TS methods. }  \label{M2M33s4p}
\end{table}
  \begin{table}
  \begin{center} {\small
 \begin{tabular}{|lllllllllll|} \hline
 &  K = & 0.1 & 0.2 & 0.3 & .5 &  1.0 & 1.5 &1.6 & 1.8 & 2.0 \\ \hline
 (M1) &  $\sspcoef_{TS}$  = &  0.4400 & 0.6921 & 0.9662 & 1.5617 & 2.6669 & 3.4735 & 3.5607 & 3.6759 & 3.7161 \\
 &  $\ceff$ = & 0.0550  &  0.0865  &  0.1208  &  0.1952  &  0.3334  &  0.4342  &  0.4451 &   0.4595   & 0.4645 \\ \hline
 (M2) &   $\sspcoef_{TS}$  = & 0.3523 & 0.6569 & 0.9662 & 1.5617 & 2.6669 & 3.4735 & 3.5301 & 3.5850 & 3.6282 \\
 &   $\ceff$ = & 0.0440  &  0.0821  &  0.1208  &  0.1952  &  0.3334 &   0.4342 &    0.4413 &  0.4481  &  0.4535 \\\hline
 (M3) &  $\sspcoef_{TS}$  = &  0.3381 & 0.6102 & 0.8407 & 1.2174 & 1.8181 & 2.0596 & 2.0793 & 2.1030 & 2.1093  \\ 
&   $\ceff$ = & 0.0676   & 0.1220   & 0.1681  &  0.2435   & 0.3636  &  0.4119  & 0.4159  &  0.4206   &  0.4219 \\ \hline
 \end{tabular} }
\end{center}
 \caption{SSP-TS coefficients of four stage fourth order SSP-TS methods.  \label{M2M34s4p} }  
 \end{table}
 \begin{table}
 \begin{center} {\small
 \begin{tabular}{|lllllllllll|} \hline
 &  K& = 0.1 & 0.2 & 0.3 & .5 &  .6 & .7 & 1.0 & 1.5 & 2.0 \\ \hline
 (M1) &  $\sspcoef_{TS}$  = & 1.5256 & 1.5768 & 1.6563 & 2.0934 & 2.4472 & 2.7819 & 3.5851 & 4.4371 & 4.9919 \\
  & $\ceff$ = & 0.1526  &  0.1577   & 0.1656 &   0.2093 &   0.2447  &  0.2782  &  0.3585   & 0.4437  &  0.4992 \\ \hline
 (M2) &   $\sspcoef_{TS}$  = & 0.5876 & 1.0003 & 1.3319 & 2.0934 &  2.4472 & 2.7819 & 3.5381 & 4.3629    & 4.6614 \\
  &$\ceff$ = &  0.0588  &  0.1000    & 0.1332 &   0.2093 &   0.2447 &   0.2782  &  0.3538   & 0.4363  &  0.4661 \\ \hline
 (M3) & $\sspcoef_{TS}$  = &  0.5631  & 0.9296 & 1.2057 & 1.6551 & 1.8554 & 2.0300 & 2.4407 & 2.8748 & 2.9768 \\ 
 & $\ceff$ = & 0.0939  &  0.1549   & 0.2009  &  0.2758 &   0.3092  &  0.3383   & 0.4068    & 0.4791  &  0.4961 \\ \hline
 \end{tabular}}
 \end{center}
 \caption{SSP-TS coefficients of five stage fourth order methods. }  \label{M2M35s4p}
\end{table}

\smallskip
 
\noindent {\bf Four Stage SSP-TS Methods}:
 While four stage fourth order explicit SSP Runge-Kutta methods do not exist, 
four stage fourth order SSP-TS  explicit two-derivative Runge-Kutta methods do. 
Four stage fourth order methods do not necessarily satisfy the stage order two ($q=2$) condition. These methods 
 have a more nuanced behavior:  for very small $K<0.2$, the optimized SSP methods have stage order $q=1$. 
 For $0.2 < K< 1.6$ the optimized SSP methods have  stage order $q=2$.  Once $K$ becomes larger again,
 for $K \geq 1.6$, the optimized SSP methods are once again of stage order $q=1$. However, the difference in the 
 SSP coefficients is very small (so small it does not show on the graph) so the (M2) methods can be used without
 significant loss of efficiency. 
 
As seen in  Table \ref{M2M34s4p}, the methods with the special structure (M3) have smaller SSP coefficients. But
when we look at the  effective SSP-TS coefficient  we notice that, once again, for smaller $K$ they are more efficient. 
Table \ref{M2M34s4p} shows that the (M3) methods  are more efficient  when 
$K \leq 1.5 $, and remain competitive for larger values of $K$.

 It is interesting to consider the limiting case, \noindent{\bf  SSP-TS  M2($4,4,\infty$)}, in which the Taylor series formula is unconditionally stable (i.e.,  $K=\infty$). 
This provides us with an upper bound of the SSP coefficient for this class of methods by ignoring any time step constraint coming from condition \eqref{TS}. 
 A four stage fourth order method that is optimal for $K=\infty$ is:

 \[
\begin{array}{l l}
\vspace{.25 cm}
y^{(1)} &= u^n\\ \vspace{.25 cm}
y^{(2)} &= u^n + \frac{1}{4}  \Delta t F(y^{(1)}) + \frac{1}{32} \Delta t^2  \dot{F}(y^{(1)}) \\ \vspace{.25 cm}
y^{(3)} &=u^n +   \frac{1}{4} \Delta t \left( F(y^{(1)}) + F(y^{(2)}) \right) +
 \frac{1}{32} \Delta t^2  \left( \dot{F}(y^{(1)}) +   \dot{F}(y^{(2)}) \right)  \\ \vspace{.25 cm}
y^{(4)}&= u^n + \frac{1}{4} \Delta t \left( F(y^{(1)})  +F(y^{(2)})  +F(y^{(3)})  \right)
+ \frac{1}{32} \Delta t^2 \left( \dot{F}(y^{(2)}) +2  \dot{F}(y^{(3)}) \right) \\  \vspace{.25 cm}
u^{n+1} & = u^n + \frac{1}{4} \Delta t \left( F(y^{(1)})  +F(y^{(2)})  +F(y^{(3)}) +F(y^{(4)})   \right) \\
& + \frac{1}{288} \Delta t^2 \left( 
5 \dot{F}(y^{(1)}) +
12 \dot{F}(y^{(2)}) +
3 \dot{F}(y^{(3)})+
16 \dot{F}(y^{(4)})
   \right). \\ 
\end{array}
\] 
This method has SSP coefficient $\sspcoef_{TS}=4$, with effective SSP coefficient $\ceff = \frac{1}{2}$.
This method also has stage order $q=2$.
This method is not intended to be useful in the SSP context but gives us an idea of the limiting behavior: i.e., what the best possible value of 
$\sspcoef_{TS}$ could be
if the Taylor series condition had no constraint ($K=\infty)$. We observe in Table \ref{M2M34s4p} that the SSP coefficient of the M2(4,4,K) method  is within 
10\% of this limiting $\sspcoef_{TS}$ for values of $K=2$.

\smallskip
   
\noindent {\bf Five Stage Methods}:
The optimized five stage fourth order methods have stage order $q=2$ for the values of $0.5 \leq K \leq 7$, and 
otherwise have stage order $q=1$. The SSP coefficients of these methods are shown in the green line in Figure \ref{Fig4thO},
and the SSP and effective SSP coefficients for all three types of methods are compared in Table \ref{M2M35s4p}.
We observe that these methods have higher effective SSP coefficients than the corresponding four stage methods.

 \subsection{ {Fifth Order SSP-TS Methods}}
 While fifth order explicit SSP Runge-Kutta methods do not exist, the addition of a second derivative which satisfies the Taylor Series 
 condition allows us to find explicit SSP-TS methods of fifth order. 
 For fifth order, we have the result (in Section  \ref{OrderBarriers} above) that  all methods must satisfy the stage order $q=2$ condition,
so we consider only (M2) and (M3) methods. In Figure \ref{Fig5thO} we show the SSP-TS  coefficients of M2(s,5,K) methods for $s=4,5,6$.
 
 \smallskip
 
\noindent {\bf Four Stage Methods}:
Four stage fifth order methods exist, and their SSP-TS coefficients are shown in blue in Figure \ref{Fig5thO}. 
 We were unable to find M3(4,5,K) methods, possibly due to the paucity of  available coefficients for this form. 
 
 \smallskip
 
\noindent {\bf Five Stage Methods}:
The SSP coefficient of the five stage M2 methods can be seen in red in Figure \ref{Fig5thO}.  We observe that 
 the SSP coefficient of the M2(5,5,K) methods plateaus with respect to $K$.
As shown in Table \ref{fifthO}, methods with the form (M3) have a significantly smaller SSP coefficient than that of (M2). 
However, the effective SSP coefficient is more informative here, and we see that the (M3) methods are
 more efficient for small values of $K \leq 0.5$, but not for larger values.

  \smallskip
 
\noindent {\bf Six Stage Methods}:
The SSP coefficient of the six stage M2 methods can be seen in green in Figure \ref{Fig5thO}. 
In Table  \ref{fifthO} we compare the SSP coefficients and effective SSP coefficients of (M2) and (M3) methods.
As in the case above,  the methods with the form (M3) have a significantly smaller SSP coefficient than that of (M2), and 
  the SSP coefficient of the (M3) methods plateaus with respect to $K$. However, the effective SSP coefficient shows that
the (M3) methods are more efficient for small values of $K \leq 0.7$, but not for larger values.

 \begin{table}[h]
  \begin{center} {\small
 \begin{tabular}{|lllllllllll|} \hline
    &  K =& 0.1 & 0.2 & 0.3 & 0.5  & 1.0 & 1.5 & 1.6 & 1.8 & 2.0 \\ \hline
M2(5,5,K) &$\sspcoef_{TS}$  = &   0.3802 & 0.7448 & 1.0892 & 1.6877 & 2.9281 & 3.8102 & 3.8479 & 3.8879 & 3.8971 \\
&    $\ceff$ = & 0.0380  &  0.0745  &   0.1089  &  0.1688  &  0.2928  &  0.3810  & 0.3848  &  0.3888   & 0.3897 \\ \hline
M3(5,5,K) &  $\sspcoef_{TS}$  = & 0.3298 & 0.5977 & 0.8186 & 1.0625 &  1.0625 & 1.0625 & 1.0625 & 1.0625 & 1.0625 \\ 
&    $\ceff$ = &  0.0550  &  0.0996  &   0.1364   & 0.1771  &  0.1771  &  0.1771  & 0.1771  &  0.1771   & 0.1771 \\ \hline
M2(6,5,K) & $\sspcoef_{TS}$  = &0.5677  & 1.0230 & 1.4581 & 2.2102 & 3.8749 & 4.9201 & 5.0002 & 5.0903 & 5.1301 \\   
 & $\ceff$  = & 0.0473  &  0.0852  &  0.1215  &  0.1842  &  0.3229  &  0.4100  &  0.4167 &   0.4242  &  0.4275 \\\hline 
M3(6,5,K) &$\sspcoef_{TS}$  = & 0.5398  & 0.9370 & 1.2592 & 1.6914 & 1.8208 & 1.8208 & 1.8208 & 1.8208 & 1.8208 \\ 
 &   $\ceff$ =&  0.0771  &  0.1339  &  0.1799  &  0.2416  &  0.2601  &  0.2601  &  0.2601  &  0.2601   & 0.2601 \\ \hline
   \end{tabular}}
 \end{center}
  \caption{SSP-TS coefficients and effective SSP-TS coefficients of fifth order methods. }  \label{fifthO}
\end{table} 

\begin{table}[h]
     \begin{center}
 \begin{tabular}{|lllllllll|} \hline
   & K =     &   0.1 & 0.2 & 0.3 & 0.5  & 1.0 & 1.5   & 2.0 \\ \hline
M2(5,6,K) & $\sspcoef_{TS}$  &  0.1441 & 0.2280 & 0.2780 & 0.3242 & 0.3500 & 0.3536 & 0.3555 \\
& $\ceff$  &   0.0144 &   0.0228 &   0.0278  &  0.0324  &  0.0350 &   0.0354 &   0.0355 \\ \hline
M2(6,6,K)  & $\sspcoef_{TS}$  &  0.2944 & 0.5157 & 0.6725 & 0.9044 & 1.5225 & 2.0002 & 2.1966 \\ 
& $\ceff$  &   0.0245 &  0.0430  &  0.0560    &  0.0754 &   0.1269 &   0.1667 &   0.1831 \\ \hline
M2(7,6,K)    & $\sspcoef_{TS}$ &  0.3981 &  0.7158 & 0.9734 & 1.4217 & 2.0376 & 2.5648 & 2.7794 \\
   &  $\ceff$ & 0.0284  &  0.0511  &  0.0695  &  0.1016   &  0.1455   & 0.1832   & 0.1985 \\ \hline 
M3(7,6,K)    &$\sspcoef_{TS}$ &  0.3547 & 0.6007  & 0.8059  & 0.8941 & 0.8947  & 0.8947  & 0.8947  \\ 
     &  $\ceff$ &  0.0443 &   0.0751 &   0.1007  &  0.1118  &  0.1118  &  0.1118  &  0.1118 \\ \hline
M3(8,6,K) & $\sspcoef_{TS}$   & 0.5495 & 0.9754 &1.2882 & 1.6435 & 1.7369 &1.7369 &1.7369  \\ 
& $\ceff$   &  0.0611  &  0.1084  &  0.1431   & 0.1826  &   0.1930  &  0.1930 &    0.1930 \\ \hline
       \end{tabular}
 \end{center}
 \caption{SSP-TS coefficients and effective SSP-TS coefficients of sixth orderSSP-TS methods.}
          \label{sixthO}
     \end{table}

\vspace{-.1in}
 \subsection{ {\bf Sixth Order SSP-TS Methods}}
As shown in  Section \ref{OrderBarriers}, the highest order of accuracy this class of methods can obtain is $p=6$, and these methods must satisfy \eqref{SO2}. 
 We find sixth order methods with $s=5,6,7$ stages of type (M2).  As to methods with the special structure M3, we are unable to find methods with $s \leq p$,
 but we find M3(7,6,K) methods and M3(8,6,K) methods.
  In the first six rows of Table \ref{sixthO} we compare the SSP-TS and effective SSP-TS coefficients of the (M2) methods  with $s=5,6,7$ stages.
 In the last four rows of  Table  \ref{sixthO}  we compare the SSP coefficients and effective SSP coefficients for sixth order methods with $s=7,8$ stages.
  Figure \ref{Fig6thO} shows the SSP-TS coefficients of the optimized (M3) methods for seven and eight stages, which clearly plateau with respect to $K$
  (as can be seen in the tables as well). 
 For the sixth order methods, it is clear that M3(8,6,K) methods are most efficient for all values of $K$.

 \begin{figure}[t]
   \centering
     \begin{minipage}{0.475\textwidth}
        \centering
        \includegraphics[width=0.9\linewidth]{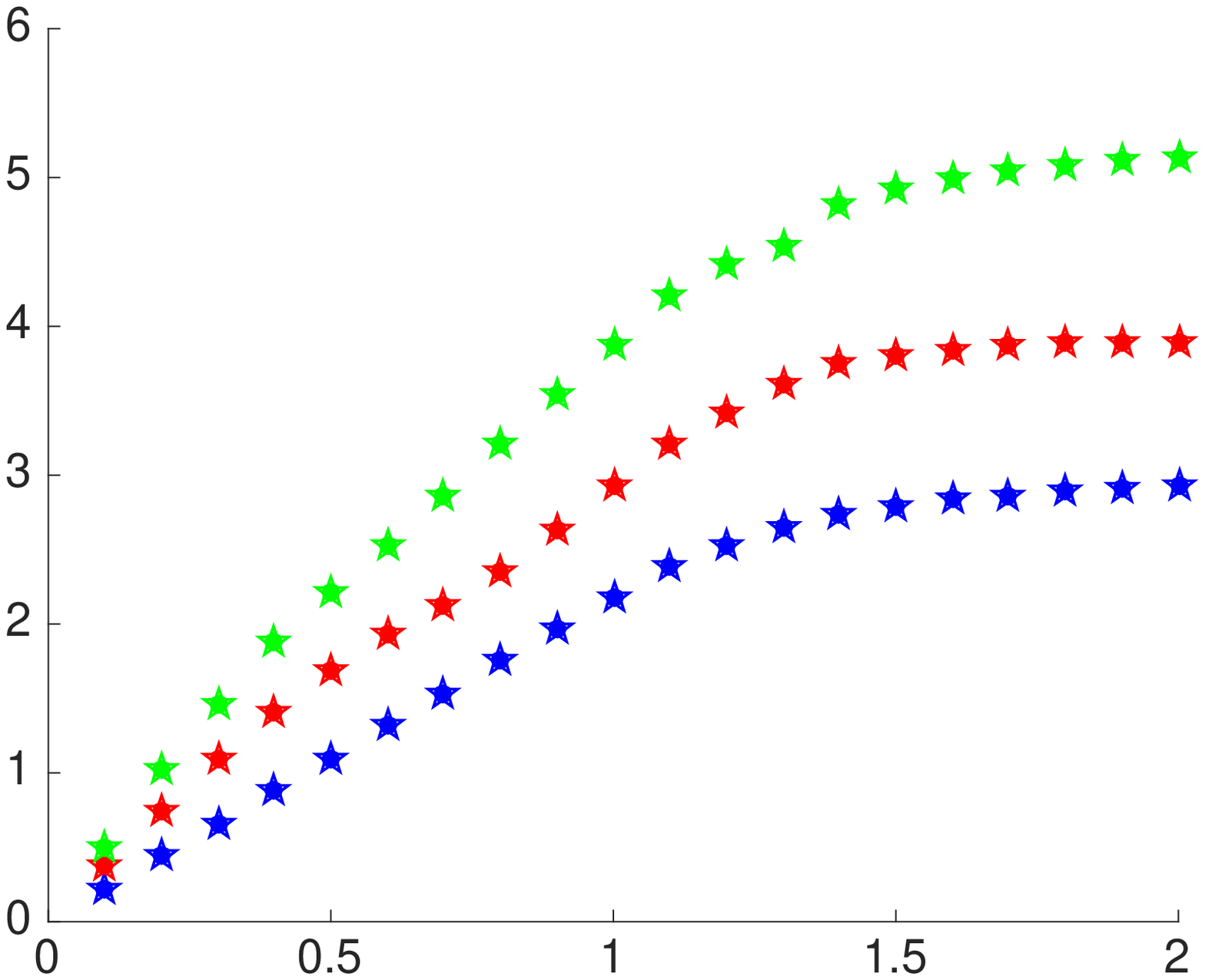}
\caption{SSP-TS coefficients  $\sspcoef_{TS}$ (on y-axis) for  \textcolor{blue}{M2(4,5,K)} and \textcolor{red}{M2(5,5,K)}
and \textcolor{green}{M2(6,5,K), where $K$ is on x-axis}.
\label{Fig5thO} }
    \end{minipage}%
    \hspace{0.2in}
    \begin{minipage}{0.475\textwidth}
        \centering
        \includegraphics[width=0.9\linewidth]{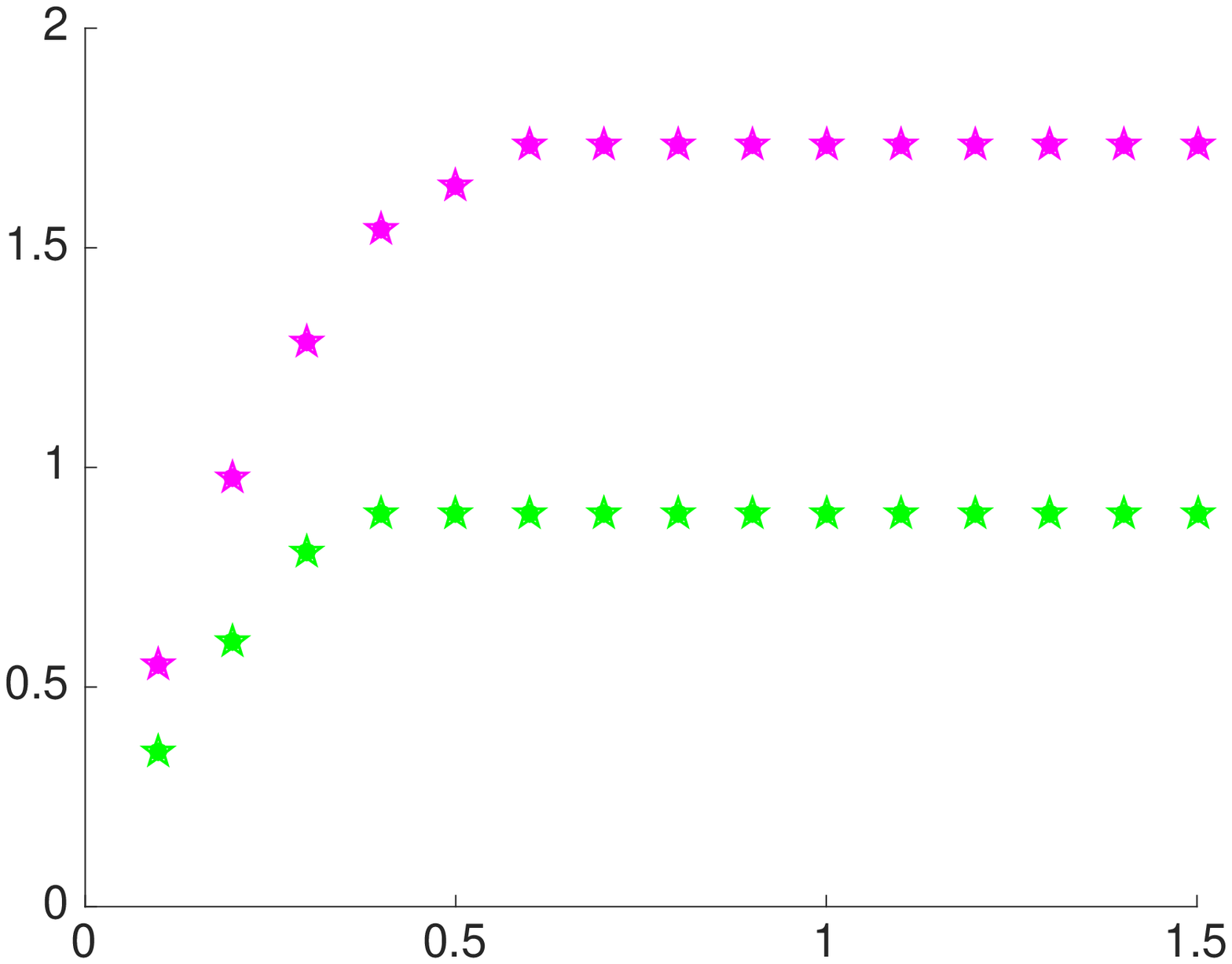}
\caption{SSP-TS coefficients $\sspcoef_{TS}$  (on y-axis) for  \textcolor{green}{M3(7,6,K)} and \textcolor{magenta}{M3(8,6,K),
where $K$ is on x-axis}
}
\label{Fig6thO}
\end{minipage}
\end{figure}

\subsection{Comparison with existing methods}
First, we wish to compare the methods in this work to those in our prior work \cite{MSMD}. 
If a spatial discretization satisfies the forward Euler condition \eqref{FE} and the second derivative condition \eqref{SD} it will
also satisfy the Taylor series condition \eqref{TS}, with \[ K = \tilde{K} \left( \sqrt{\tilde{K}^2 +2    } - \tilde{K}  \right).\] 
In this case, it is preferable to use the SSP-SD MDRK methods in \cite{MSMD}. However, in the case that 
the second derivative condition  \eqref{SD}  is not satisfied for any value of $\tilde{K} >0$, or if 
the Taylor series condition is independently satisfied with a larger $K$ than would be established from the two conditions, i.e.,
$ K > \tilde{K} \left( \sqrt{\tilde{K}^2 +2    } - \tilde{K}  \right)$, then it may be preferable to use one of the SSP-TS methods 
derived in this work.

Next, we wish to compare the methods in this work to those in \cite{Nguyen-Ba2010}, which was 
the first paper to consider an SSP property based on the forward Euler and Taylor series base conditions.
The approach  used in our work is somewhat similar to that in \cite{Nguyen-Ba2010} where the authors consider building time integration schemes which can be composed as convex combinations of forward Euler and Taylor series time steps, 
where they aim to find methods which are optimized for the largest SSP coefficients.  However, there are several differences between our approach and the one of \cite{Nguyen-Ba2010}, which results in the fact that in this paper we
are able to  find more methods, of higher order, and with better SSP coefficients.  
In addition, in the present work we find and prove an order barrier for SSP-TS methods.

The first difference between our approach and the approach in \cite{Nguyen-Ba2010} is that we allow computations of $\dot{F}$ of the intermediate values, 
rather than only $\dot{F}(u^n)$. 
Another way of saying this is that we consider SSSP-TS
methods that are not of type M3, while the methods considered in  \cite{Nguyen-Ba2010} are all of type M3. 
In some cases, when we restrict our search to M3 methods and $K=1$, we find methods with the same SSP coefficient as in \cite{Nguyen-Ba2010}. For example, 
HBT34  matches our SSP-TS M3(3,4,1) method with an SSP coefficient of $\sspcoef_{TS}=1$,
HBT44 matches our SSP-TS M3(4,4,1) method with $\sspcoef_{TS}=\frac{20}{11}$,
HBT54 matches our  SSP-TS M3(5,4,1)   method with $\sspcoef_{TS}=2.441$,
 and HBT55 matches our SSP-TS M3(5,5,1) method with an SSP coefficient of  $\sspcoef_{TS}=1.062$.
While methods of type M3 have their advantages, they are sometimes sub-optimal in terms of efficiency,
as we point out in the tables above. 

The second difference between the SSP-TS methods in this paper  and the methods in \cite{Nguyen-Ba2010}
is that in \cite{Nguyen-Ba2010} only one  method of order $p>4$ is reported, 
while we have many fifth and sixth order methods of various types and stages, optimized for a variety of $K$ values.

The most fundamental difference between our approach and the approach in \cite{Nguyen-Ba2010} is 
that our methods are optimized for the relationship between the forward Euler restriction and the Taylor series restriction while
the  time step restriction in the methods of  \cite{Nguyen-Ba2010}   is defined as the most restrictive of the forward Euler and Taylor series time step conditions.   
Respecting the minimum of the two cases will still satisfy the nonlinear stability property, 
 but this approach does not allow for a balance between the restrictions  considered, which can lead to severely more restrictive conditions. 
  In our approach we use the relationship between the two time-step
 restrictions to select optimal methods. For this reason, the methods we find have larger allowable time-steps in many cases. 
 To understand this a little better consider the case where the forward Euler condition is $\dt_{FE} \leq \dx$ and the 
 Taylor series condition is $\dt_{TS} \leq \frac{1}{2}\dx$. In the approach used in \cite{Nguyen-Ba2010}, the
base time step restriction is then $\dt_{max} = \max \{ \dt_{FE}, \dt_{TS} \} \leq \frac{1}{2}\dx$.  
The  HBT23 method in \cite{Nguyen-Ba2010} is a third order scheme with 2 stages which has a SSP coefficient of $\sspcoef_{TS}=1$,
so  the allowable time-step with this scheme will be the same  $\dt \leq  \sspcoef_{TS} \dt_{max} \leq \frac{1}{2}\dx$. 
On the other hand,  using our optimal SSP-TS M2(2,3,.5)  scheme, which has an SSP coefficient $\sspcoef_{TS}=.75$,
the allowable  time step is $\dt \leq \sspcoef_{TS} \dt_{FE}  \leq \frac{3}{4} \dx$, a 50\% increase.  
This is not only true when $K<1$: consider the case where $\dt_{FE}\leq \frac{1}{2}\dx$ 
and $\dt_{TS} \leq \dx$. Once again the  HBT23 method in \cite{Nguyen-Ba2010} 
 will have a time step restriction of $\dt \leq  \sspcoef_{TS} \dt_{max} \leq \frac{1}{2}\dx$, while our M2(2,3,2) method
has an SSP coefficient $\sspcoef_{TS}=1.88$, so that the overall time step restriction would be $\dt \leq \frac{1.88}{2} \dx =.94 \dx$, which is 88\% larger. 
 Even when the two base conditions are the same (i.e., $K=1$)
and we have  $\dt_{FE} \leq \dx$ and $\dt_{TS} \leq \dx$, the HBT23 method in \cite{Nguyen-Ba2010} gives an allowable time-step 
of $\sspcoef_{TS}=1$ while our SSP-TS M2(2,3,1) has an SSP coefficient $\sspcoef_{TS}=1.5$, so that our
method allows a time-step that is 50\% larger.\footnote{These efficiency measures do not account for the fact that the methods
 in \cite{Nguyen-Ba2010} are of type SSP-TS M3 and so require fewer funding evaluations. 
Correcting for this, our methods are still 10-40\% more efficient.}
These simple cases demonstrate that our methods, which are optimized for the value of $K$, will usually allow a larger SSP coefficient that  the methods obtained in \cite{Nguyen-Ba2010}.


%

\section{Numerical Results}
\subsection{Overview of numerical tests}
We wish to test our methods on what are now considered standard benchmark tests in the SSP community.  
In this subsection we preview our results, which we then present
in more detail throughout the remainder of the section.

First, in the tests in  Section \ref{Example1}  we focus on how the strong stability properties of these methods are observed in practice, by considering the
 total variation of the numerical solution. We focus on two scalar PDEs: the linear advection equation and Burgers' equation, using simple first order spatial discretizations 
which are known to  satisfy a  total variation diminishing property over time for the forward Euler and Taylor series building blocks.
We want to ensure that our numerical approximation to these solutions observe similar properties 
as long as the predicted SSP time step restriction, $\Delta t \leq \sspcoef_{TS} \Delta t_{FE}$, is respected.  
These scalar one-dimensional partial differential equations are chosen for their simplicity so we may understand the behavior of the numerical solution, 
but the discontinuous initial conditions may lead to instabilities if standard time discretization techniques are employed.
Our tests  show that the methods we design here  preserve these properties as expected by the theory. 

In Example 2, we extend the results from Example 1 to the case where we use the higher order weighted essentially non-oscillatory 
 (WENO) method, which is not provably TVD but gives results that have very small increases in total variation.
We demonstrate that our methods out-perform other methods, such as the SSP-SD MDRK methods in \cite{MSMD}, and that non-SSP methods
that are standard in the literature do not preserve the TVD property for any time-step.

In many of these examples we are concerned with the total variation diminishing property.
To  measure the sharpness of the SSP condition we compute the maximal observed rise in total variation over each step, defined by
\begin{equation}
\label{eqn:increase-in-tv}
	\max_{0 \leq n \leq N-1} \left( \| u^{n+1} \|_{TV} - \| u^n \|_{TV} \right), 
\end{equation}
as well as the maximal observed rise in total variation over each stage, defined by
\begin{equation}
	\max_{1 \leq j \leq s} \left( \| y^{(j+1)} \|_{TV} - \| y^{(j)}  \|_{TV} \right) , 
\end{equation}
where  $y^{(s+1)}$ corresponds to $u^{n+1}$.
The quantity of interest is the time-step $\dt_{obs}$, or the SSP coefficient  $\sspcoef_{TS}^{obs} = \frac{\dt_{obs}}{\DtFE}$  at which this rise becomes significant, as defined by a maximal increase of $10^{-10}$.

It is important to notice that theSSP-TS methods we designed depend on the value of $K$ in \eqref{TS}.
However, in practice we often do not know the exact value of $K$. In Example 3 we investigate what happens when we use
spatial discretizations with a given value of $K$ with time discretization methods designed for an incorrect value of $K$.
We conclude that although in some cases a smaller step-size is required, for methods of type M3 there is generally no
adverse result from selecting the wrong value of $K$.

In Example 4 we investigate the increased flexibility in the choice of spatial discretization
that results from relying on the \eqref{FE} and \eqref{TS} base conditions. The 
only  constraint in the choice of differentiation operators $D_x$ and $\tilde{D}_x$ 
(described at the end of Section 1.2) is that the resulting 
building blocks must satisfy the monotonicity conditions \eqref{FE} and \eqref{TS}  in the desired convex functional $\| \cdot \|$.
As noted above, this constraint is less restrictive than requiring that \eqref{FE} and \eqref{SD} are satisfied: any spatial discretizations
for which \eqref{FE} and \eqref{SD}  are satisfied will also satisfy  \eqref{TS}. However, there are some 
spatial discretizations that satisfy  \eqref{FE} and \eqref{TS} that do not satisfy \eqref{SD}. In Example 4 
we find that choosing spatial discretizations that satisfy \eqref{FE} and \eqref{TS}  but not \eqref{SD} allows
for larger time-steps before the rise in total variation.
And finally, in Example 5, we demonstrate the positivity preserving behavior of our methods 
when applied to a nonlinear system of equations.

\subsection{On the numerical implementation of the second derivative.}

In the following numerical test cases the spatial discretization  is performed as follows: 
at each iteration we take the known value $u^n$ and compute the flux
$f(u^n) = - u^n$ in the linear case and $f(u^n) = \frac{1}{2} \left(u^n \right)^2$ for Burgers' equation. 
Now to compute the  spatial derivative $f(u^n)_x$ we use an operator $D_x$ and compute 
\[ u^n_t = - f(u^n)_x    \;\;\;\;  \rightarrow \;\;\;\; u^n_t =D_x(- f(u^n)).   \]
In the numerical examples below the differential operator $D_x$ will represent, depending on the problem,
 a first order upwind finite difference scheme and the fifth order finite difference WENO 
method \cite{jiang1996}. In our scalar  test cases $f'(u)$ does not change sign, so we avoid 
 flux splitting.
 
 Now we have the approximation to $U_t$ at time $t^n$, and wish to compute the approximation to $U_{tt}$.  
For the linear advection problem, this is very straightforward as $U_{tt} =  U_{xx}$. To compute this, we take $u_x$ as computed before, 
and differentiate it again. 
For Burgers' equation, we have $U_{tt} =  \left(- U U_t \right)_x$.  We take the approximation to $U_t$ that we 
obtained above,
and we multiply it by $u^n$, then differentiate in space once again. 
 In pseudocode, the calculation takes the form:
\begin{eqnarray*}
u^n_{tt} =(- f'(u^n)  u^n_t)_x    \;\;\;\; &\rightarrow \;\;\;\; u^n_{tt} = \tilde{D}_x(- f'(u^n)  u^n_t  ). 
\end{eqnarray*}
Using these, we can now construct our two building blocks
\[ \mbox{\bf Forward Euler} \; \; \; u^{n+1} = u^n +\dt u^n_t ,\]
\[ \mbox{\bf Taylor series} \; \; \; u^{n+1} = u^n +\dt u^n_t +  \frac{1}{2} \dt^2 u^n_{tt}.\]
In choosing the spatial discretizarions $D_x$ and $\tilde{D}_x$ it is important that these building blocks satisfy \eqref{FE} and 
\eqref{TS} in the desired convex functional $\| \cdot \|$.

\subsection{Example 1: TVD first order finite difference approximations}
\label{Example1} 
In this section we  use first order spatial discretizations, that are provably total variation diminishing (TVD),
coupled with a variety of time-stepping methods. We look at the maximal rise in total variation.

\noindent{\bf Example 1a: Linear advection.}
As a first test case, we consider  a linear advection problem
\begin{eqnarray} \label{linadv}
 U_t - U_x =0,
 \end{eqnarray} 
on a domain $x \in [-1,1]$, with step-function initial conditions
\begin{equation}
\label{eqn:sq_wave_ic}
	u_0(x) = \left\{ \begin{array}{ll}
		1 & \text{if}\ -\frac{1}{2} \leq x \leq \frac{1}{2}, \\
		0 & \text{otherwise},
\end{array} \right. 
\end{equation}
 and periodic boundary conditions. 
 This simple example is chosen as our experience has shown \cite{SSPbook2011} that this problem often 
demonstrates the sharpness of the SSP time-step.

\begin{table}[ht]
\begin{center}
\begin{tabular}{|l|ll|ll|} \hline
\multicolumn{5}{|c|}{Linear Advection} \\ \hline
method   &  $\sspcoef_{TS}^{pred}$ &  $\sspcoef_{TS}^{obs}$ &  $\ceff^{pred}$ &  $\ceff^{obs}$ \\ \hline
FE           &  1.0000 & 1.0000 & 1.00 & 1.00 \\ 
TS           &  1.0000 & 1.0000 & 0.50 & 0.50\\ \hline 
M2(3,4,1)&  1.8788 & 1.8788 & 0.31 & 0.31 \\ 
M3(3,4,1)& 1.0000 & 1.0000 & 0.25 & 0.25 \\
M2(4,4,1)&  2.6668 &  2.6668 & 0.33 & 0.33 \\
M3(4,4,1)& 1.8181 & 1.8181 & 0.36 & 0.36 \\
M2(5,4,1)&  3.5381 &  3.6291 & 0.35 & 0.36 \\ 
M3(5,4,1)& 2.4406 & 2.4406 & 0.40 & 0.40 \\ \hline
M2(4,5,1)&  2.1864 & 2.2239 & 0.27 & 0.27 \\
M2(5,5,1)&  2.9280 & 3.1681 & 0.29 & 0.31 \\
M3(5,5,1)& 1.0625 & 1.5710 & 0.17 & 0.26 \\ 
M2(6,5,1)&  3.8749 & 3.8749 & 0.32 & 0.32  \\ 
M3(6,5,1)& 1.8207 & 1.9562 & 0.26 & 0.27 \\ \hline
M2(5,6,1)& 0.3500 & 1.9398 & 0.03 & 0.19 \\
M2(6,6,1)& 1.5225 & 2.3548 & 0.12 & 0.19 \\
M2(7,6,1)& 2.1150 & 2.3695 & 0.15 & 0.19 \\  
M3(7,6,1)& 0.8946 & 1.3207 &  0.11& 0.16 \\
M3(8,6,1)&  1.7369& 1.9861 & 0.19 & 0.22 \\  \hline
\end{tabular}
\begin{tabular}{|l|ll|ll|} \hline
\multicolumn{5}{|c|}{Burgers'} \\ \hline
method   &  $\sspcoef_{TS}^{pred}$ &  $\sspcoef_{TS}^{obs}$ &  $\ceff^{pred}$ &  $\ceff^{obs}$ \\ \hline
FE           &  1.0000 & 1.0000 & 1.00 & 1.00 \\ 
TS           &  1.0000 & 1.0000 & 0.50 & 0.50\\ \hline 
M2(3,4,1)&  1.8788 & 1.8788 & 0.31 & 0.31 \\ 
M3(3,4,1)& 1.0000 & 1.0000 & 0.25 & 0.25 \\
M2(4,4,1)&  2.6668 &  2.6668 & 0.33 & 0.33 \\
M3(4,4,1)& 1.8181 & 1.8181 & 0.36 & 0.36 \\
M2(5,4,1)&  3.5381 &  3.6102 & 0.35 & 0.36 \\ 
M3(5,4,1)& 2.4406 & 2.4406 & 0.40 & 0.40 \\ \hline
M2(4,5,1)&  2.1864 & 2.2130 & 0.27 & 0.27 \\
M2(5,5,1)&  2.9280 & 3.1009 & 0.29 & 0.31 \\
M3(5,5,1)& 1.0625 & 1.5436 & 0.17 & 0.25 \\ 
M2(6,5,1)&  3.8749 & 3.8749 & 0.32 & 0.32  \\ 
M3(6,5,1)& 1.8207 & 2.0003 & 0.26 & 0.28 \\ \hline
M2(5,6,1)& 0.3500 & 1.9239 & 0.03 & 0.19 \\
M2(6,6,1)& 1.5225 & 2.2875 & 0.12 & 0.19 \\
M2(7,6,1)& 2.1150 & 2.3189 & 0.15 & 0.16 \\  
M3(7,6,1)& 0.8946 & 1.2893 &  0.11& 0.16 \\
M3(8,6,1)& 1.7369 & 1.9734 & 0.19 & 0.21 \\  \hline
\end{tabular}
\caption{\label{tab:ex1} Example 1:  $\sspcoef_{TS}^{pred}$ and $\sspcoef_{TS}^{obs}$  for SSP-TS M2 and M3 methods.}
\end{center}
\end{table}

For the spatial discretization we use a first order forward difference for the first and second derivative:
\[
	F(u^n)_j := \frac{u^n_{j+1}-u^n_j}{\Delta x} \approx U_x( x_j ), \; \; \; \; \; \;
		\mbox{and} \; \; \; \; \; \;
	\tilde{F}(u^n)_j \approx \tilde{F}(u^n)_j  := \frac{u^n_{j+2}- 2 u^n_{j+1} + u^n_{j}}{\Delta x^2} \approx  U_{xx}( x_j ).
\]
These spatial discretizations satisfy:
\begin{description}
\item{\bf Forward Euler condition} \hspace{0.25in}
$ u^{n+1}_j = u^n_j + \frac{\Delta t}{\Delta x} \left( u^n_{j+1} - u^n_j \right)$
 is  TVD for $ \Delta t \leq  \Delta x $, \\
 \hspace*{-.35in} and
\item{\bf Taylor series  condition} \hspace{0.25in}
$ u^{n+1}_j = u^n_j + \frac{\Delta t}{\Delta x} \left( u^n_{j+1} - u^n_j \right) 
+ \frac{1}{2} \left(  \frac{\dt}{\dx} \right)^2 \left( u^n_{j+2} - 2 u^n_{j+1} + u^n_{j} \right) $
 is TVD for $\Delta t \leq  \Delta x $.
\end{description}
So that $\DtFE= \Delta x $ and in this case we have $K=1$ in \eqref{TS}.
Note that the second derivative discretization used above {\em does not} satisfy the
{\bf Second Derivative} condition \eqref{SD}, so that most of the  methods we devised 
in \cite{MSMD}  do not guarantee strong stability preservation for this problem.  

 For all of our simulations for this example, we use a fixed grid of $M=601$ points, for a grid size $\Delta x = \frac{1}{600}$, 
 and a time-step $\Delta t = \lambda \dx$
 where we vary $ \lambda $ from $\lambda = 0.05$ until beyond the point where the TVD property is violated. 
 We step each method forward by $N=50$ time-steps and compare the 
 performance of the various time-stepping methods  constructed earlier in this work, for $K = 1$. 
 We define the observed SSP coefficient $\sspcoef_{TS}^{obs}$ as the multiple of $\DtFE$ for which the maximal rise in total variation
 exceeds $10^{-10}$.
 
We verify that the observed  values of $\DtFE$ and $K$ match the predicted values, and  test this problem 
to see how well the observed SSP coefficient $\sspcoef_{TS}^{obs}$ matches the predicted  SSP coefficient $\sspcoef_{TS}^{pred}$ 
for the fourth, fifth, and sixth order methods. The results are listed in the left-hand columns of Table \ref{tab:ex1}.

\noindent{\bf Example 1b: Burgers'  equation}
We repeat the example above with all the same parameters but for the problem
\begin{eqnarray} \label{burgers}
U_t + \left( \frac{1}{2} U^2 \right)_x =0
\end{eqnarray}
on $x \in (-1,1)$.
Here we use  the spatial derivatives:
$$F(u^n)_j := - \frac{f^n_{j}-f^n_{j-1}}{\Delta x} \approx -f(U)_x( x_j ),$$ 
 and
 $$\tilde{F}(u^n)_j \approx  \tilde{F}(u^n)_j :=  - \frac{f'(u^n_{j}) F(u^n)_{j} -f'(u^n_{j-1}) F(u^n)_{j-1}}{\Delta x}  \approx  (f'(U)f(U)_x)_{x}( x_j) .$$
Using Harten's lemma we can easily show that these definitions of $F$ and $\tilde{F}$ cause the Taylor series condition to be  satisfied for 
 $\Delta t \leq \Delta x$.
The results are quite similar to those of the linear advection equation in Example 1a, as can be seen in the right-hand columns of 
Table \ref{tab:ex1}.

The results from these two studies show that the
SSP-TS methods provide  a reliable guarantee  of the allowable time-step for which the method preserves the 
strong stability condition in the desired norm. For methods of order $p=4$, we observe that the SSP coefficient is sharp: the
predicted and observed values of the SSP coefficient are identical for all the fourth order methods tested.
For methods of higher order ($p=5,6$) the observed SSP coefficient is often significantly higher than the minimal value guaranteed by the theory.

\subsection{Example 2: Weighted essentially non-oscillatory (WENO) approximations}

In this section we re-consider  the nonlinear Burgers' equation \eqref{burgers}
\[ U_t + \left(\frac{U^2}{2} \right)_x =0,\]
on $x \in (-1,1)$.
We use  the step function initial conditions \eqref{eqn:sq_wave_ic},
and periodic boundaries. We use $M=201$ points in the spatial domain, so that $\Delta x =\frac{1}{100}$, and  we step forward for $N=50$ time-steps
and measure the maximal rise in total variation for each case. 

For the spatial discretization, we use the fifth order  finite difference  WENO method \cite{jiang1996} in space, 
as this is a high order method that can handle shocks. We describe this method in Appendix \ref{WENO}.
Recall that the motivation for the development of SSP multistage multi-derivative time-stepping is for use in conjunction with 
high order methods for problems with shocks. Ideally, the specially designed spatial discretizations satisfy \eqref{FE} and \eqref{TS}.
Although the weighted essentially non-oscillatory (WENO) methods do not have a theoretical guarantee of this type, 
 in practice we observe that these methods do control the rise in total variation, as long as the step-size is below a certain threshold.

Below, we refer to the WENO method on a  flux with $f'(u) \geq 0$ as WENO$^+$ defined in \eqref{WENO+} and to to the corresponding method
on a  flux with $f'(u) \leq 0$ as WENO$^-$ defined in \eqref{WENO-}. 
Because $f'(u)$ is strictly non-negative in this example, we do not need to use flux splitting,
and use $D =$WENO$^+$. For the second derivative we have the freedom to use  $\tilde{D}_x=$ WENO$^+$ or 
 $\tilde{D}_x=$ WENO$^-$.  In this example, we use $ \tilde{D}_x=D_x=$ WENO$^+$.  In Example 4 below we show that this is more
 efficient.


In Figure \ref{WENOBurgersTest} on the left, we compare the performance of our SSP-TS M3(7,5,1)  andSSP-TS M2(4,5,1) methods, 
which both have eight function evaluations per time-step, and our SSP-TS M3(5,5,1), which has 
six function evaluations per time-step, to the SSP-SD MDRK(3,5,2) in \cite{MSMD} and non-SSP RK(6,5) 
Dormand-Prince method  \cite{DormandPrince1980}, which also have
six function evaluations per time-step. 
We note that we use the SSP-SD MDRK(3,5,2) (designed for $K=2$) because this method performs best compared to other 
explicit two-derivative multistage methods
designed for different values of $K$.
Clearly, the non-SSP method is not safe to use on this example. 
The M3 methods are most efficient, allowing the largest time-step per function evaluation before the total variation begins to rise.

This conclusion is also the case for the sixth order methods. In Figure \ref{WENOBurgersTest} on the right,
we compare our SSP-TS M3(9,6,1) and M2(5,6,1) methods, which both have ten function evaluations per time-step, and our M3(7,6,1), which has 
eight function evaluations per time-step, to the SSP-SD MDRK(4,6,1) and non-SSP RK(8,6) method given 
in Verner's paper table  \cite{Verner2014}, which also have
eight function evaluations per time-step. Clearly, the non-SSP method is not safe to use on this example. 
The M3 methods are most efficient, allowing the largest time-step per function evaluation before the total variation begins to rise.

This example demonstrates the need for SSP methods: classical non-SSP methods do not control the rise in total variation. We also observe
that the methods of type M3 are efficient, and may be the preferred choice of methods for use in practice.

\begin{figure}[t!] \hspace{0.25in} 
\includegraphics[width=0.475\textwidth]{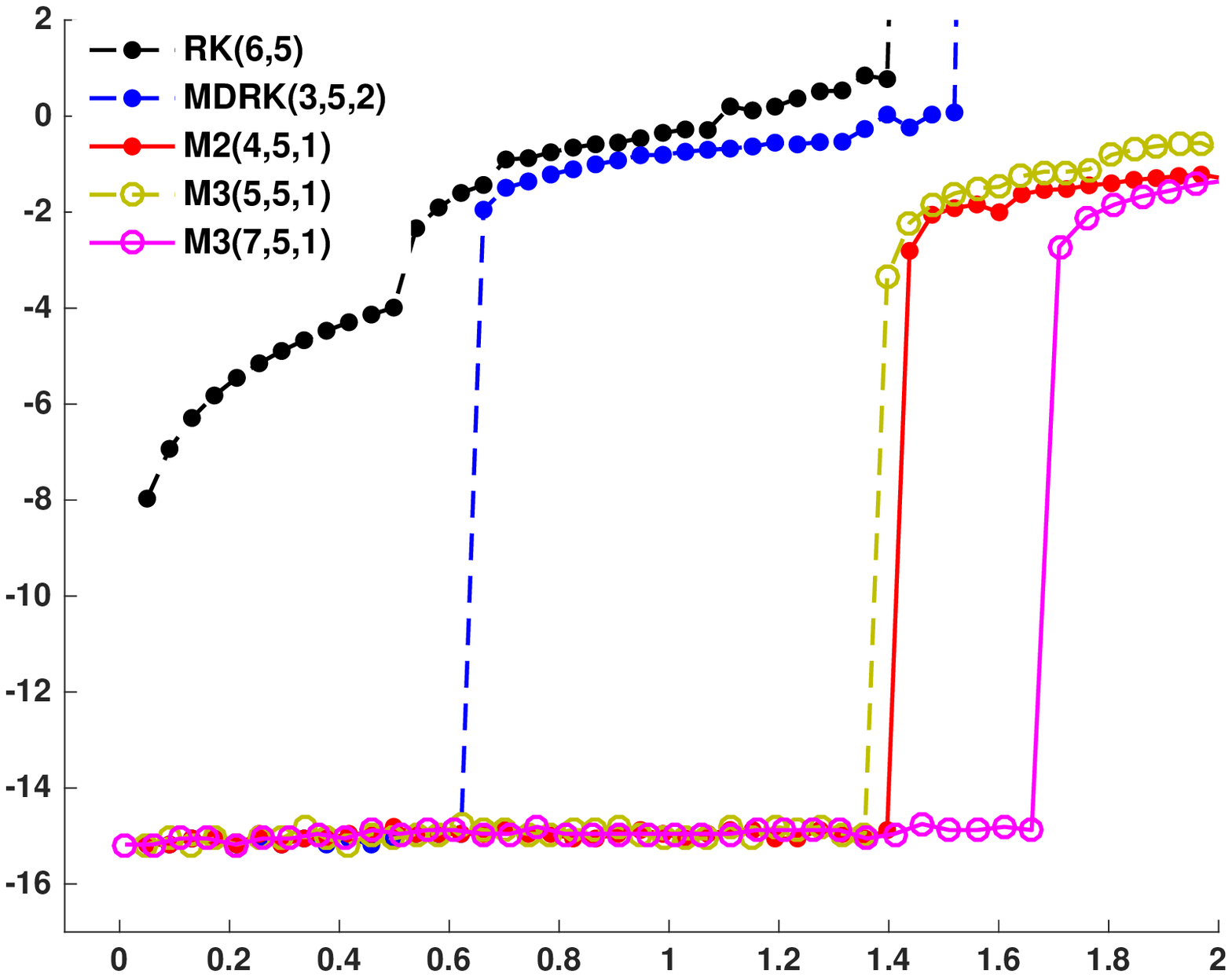} \hspace{-.35in}
\includegraphics[width=0.475\textwidth]{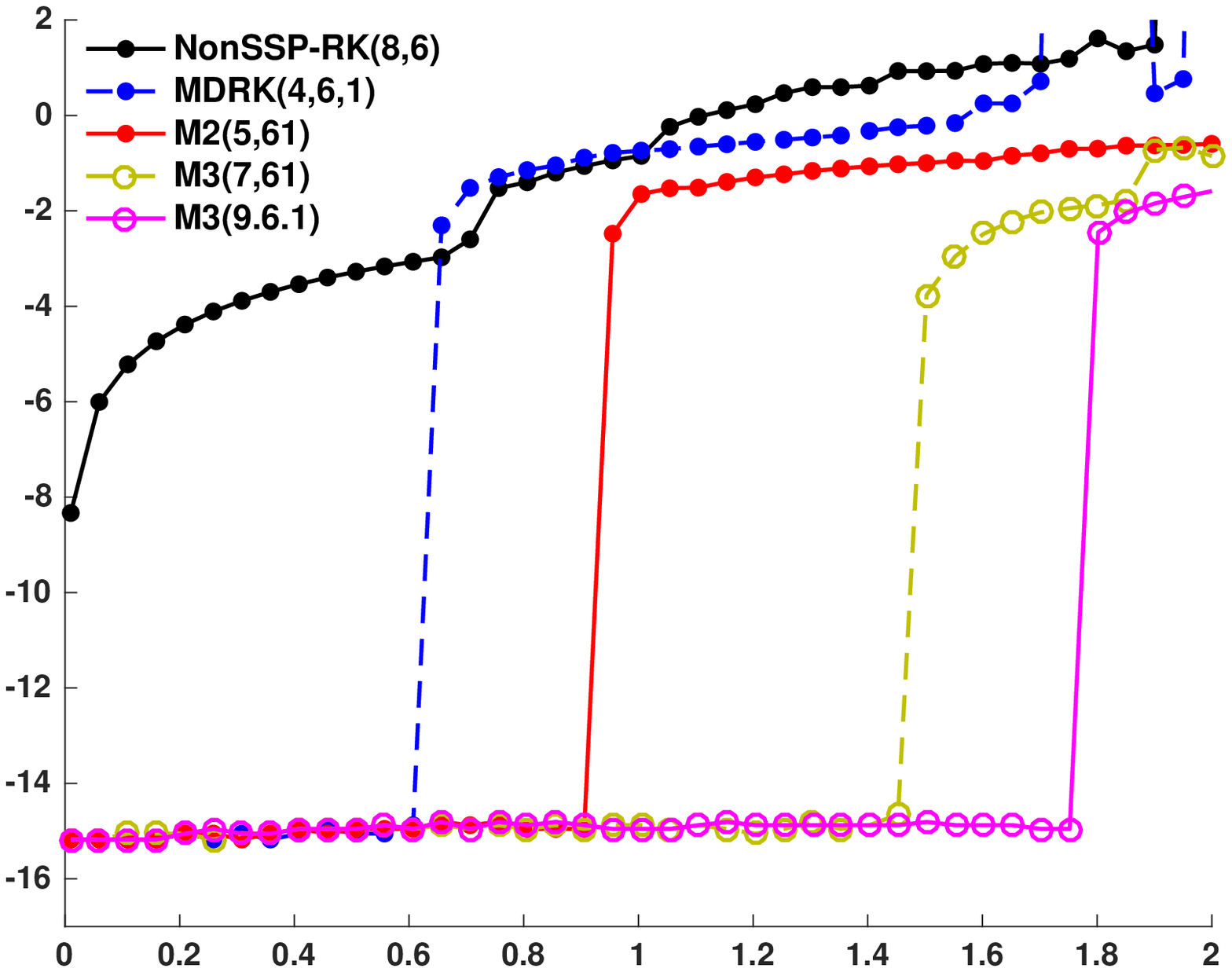} 
\caption{Example 2:
Comparison of the maximal rise in total variation (on y-axis) as a function of $\lambda=\frac{\dt}{\dx}$ (on x-axis)
 for a selection of  time-stepping methods for evolving Burgers' equation with WENO spatial discretizations.
Left: fifth order methods. Right: sixth order methods.
}
 \label{WENOBurgersTest}
\end{figure}

\subsection{Example 3: Testing methods designed with various values of  $K$}

In general, the value of $K$ is not exactly known for a given problem, so we cannot choose a method that is optimized for the correct $K$.
We wish to investigate how methods with different values of $K$ perform for a given problem.
 In this example, we  re-consider  the linear advection equation \eqref{linadv}
\[ U_t = U_x \] 
with step function initial conditions \eqref{eqn:sq_wave_ic}, and periodic boundary conditions
on $x \in (-1,1)$. We use  the fifth order WENO method  with 
$M=201$ points in the spatial domain, so that $\Delta x =\frac{1}{100}$, and  we step forward for $N=50$ time-steps
and measure the maximal rise in total variation for each case. 
Using this example, we investigate how time-stepping methods optimized for different $K$ values
perform on the linear advection with finite difference spatial approximation test case above, 
where it is known that $K=1$. We use  a variety of fifth and sixth order methods, 
designed for $0.1 \leq K \leq 2$ and give the  value of $\lambda = \frac{\dt}{\dx}$ for which the maximal rise in total variation becomes large, 
when applied to the linear advection problem.  

In Figure \ref{Kstudy} (left) we give the  observed value (solid lines)
 of $\lambda$ for a number of SSP-TS methods, M2(4,5,K), M2(5,5,K), M2(6,5,K),
M3(5,5,K), and M3(6,5,K), and the corresponding predicted value (dotted lines)
that a method designed for $K=1$ should give.  In  Figure \ref{Kstudy} (right) we repeat this study with
sixth order methods M2(5,6,K), M2(6,6,K), M3(7,6,K), and M3(8,6,K). 
We observe that while choosing the correct $K$ value can be beneficial, and is certainly important theoretically, in practice using methods 
designed for different $K$ values often makes little difference, particularly when the method is optimized for a value close to the
correct $K$ value. For the sixth order methods in particular, the observed values of the SSP coefficient are all larger than the
predicted SSP coefficient.

\begin{figure}[t!] 
\begin{center} 
\includegraphics[width=0.5\textwidth]{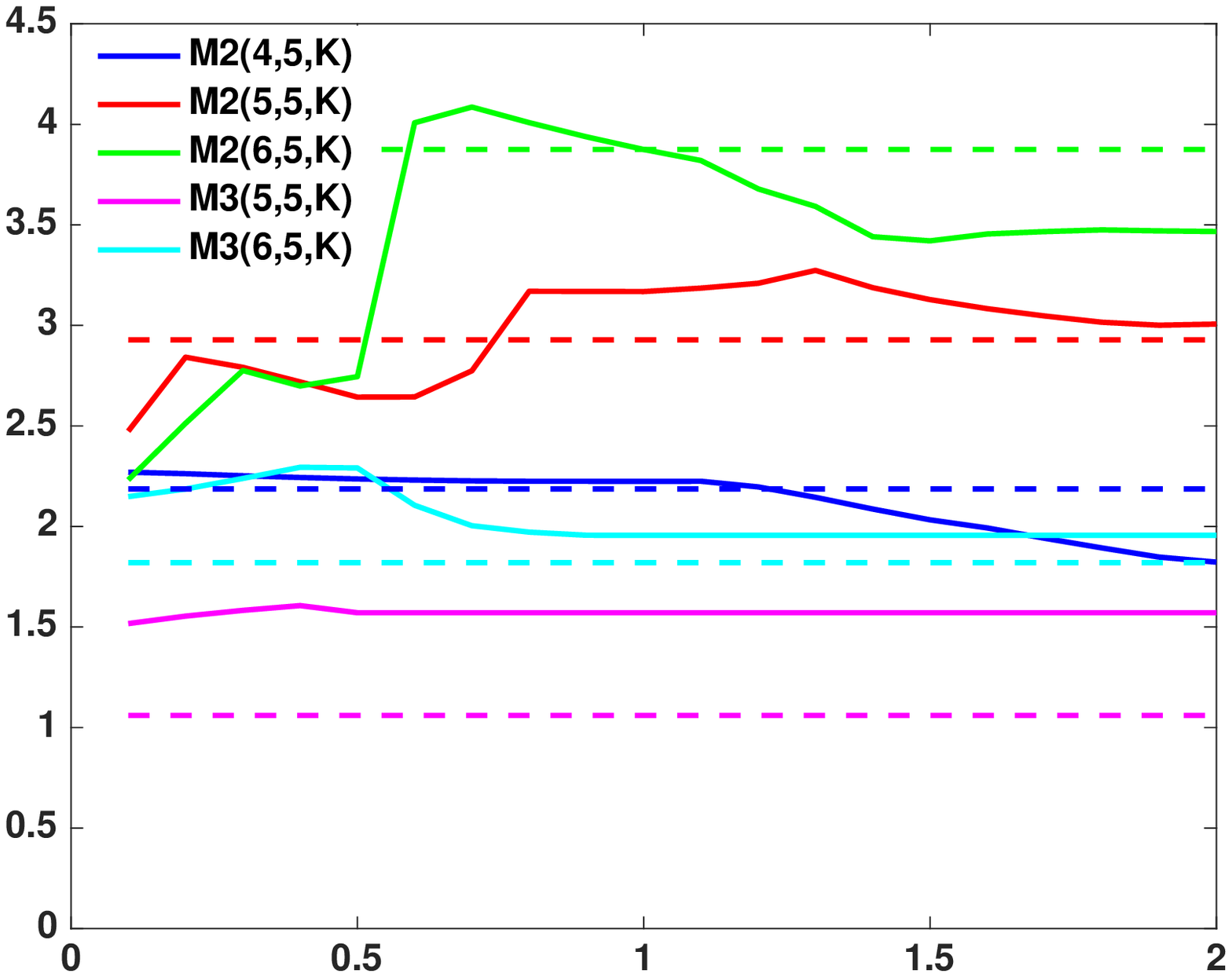}  \hspace{-0.1in}
\includegraphics[width=0.5\textwidth]{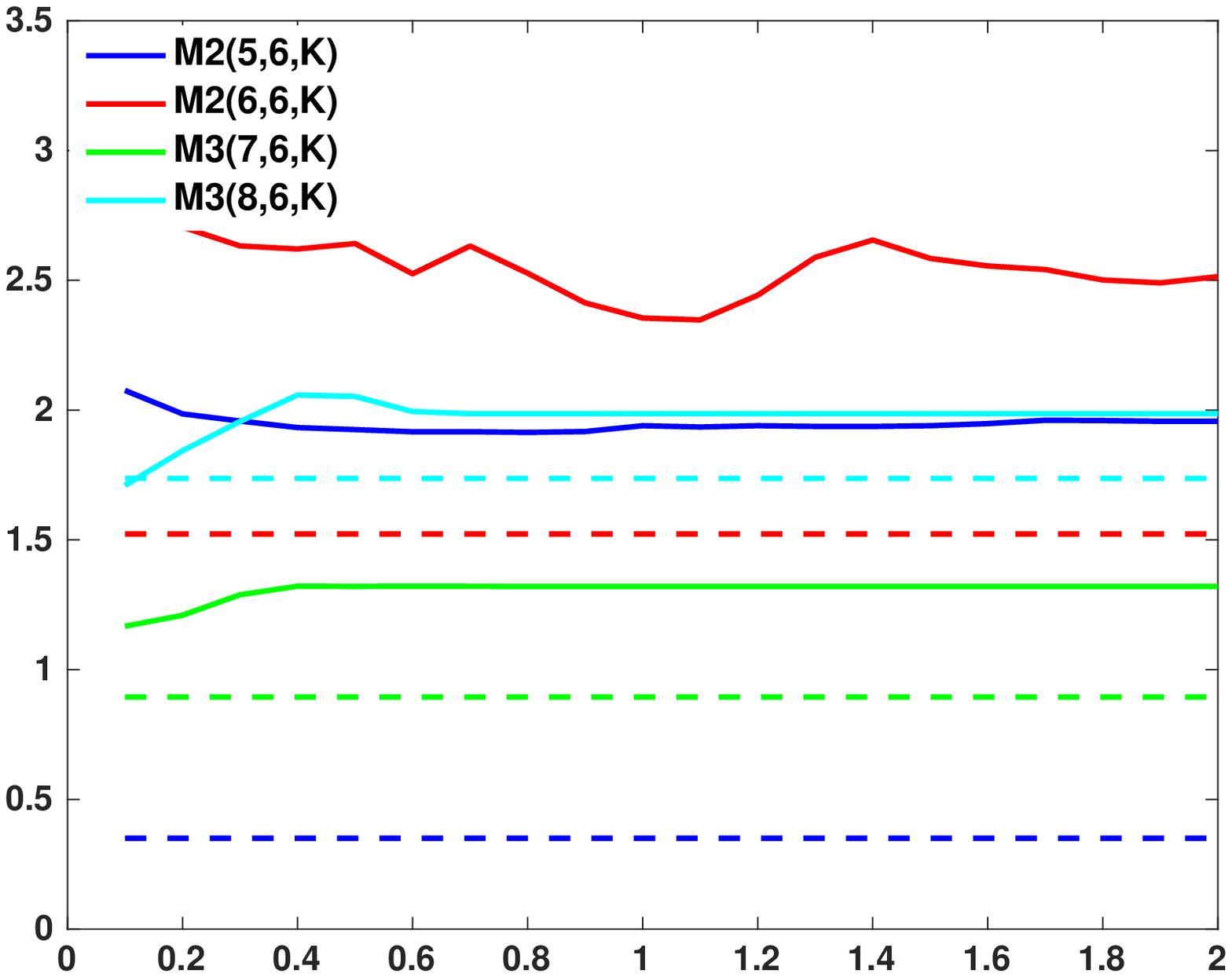} 
\end{center}
\caption{Example 3: The observed value of $\lambda=\frac{\dt}{\dx}$ such that the method is TVD (y-axis) 
when methods designed for different $K$ values (on x-axis) are applied to the problem with $K=1$. 
For each method, the observed value(solid line) is higher than the predicted value (dashed line).}
 \label{Kstudy}
\end{figure}

\subsection{Example 4: The benefit of different base conditions}

In \cite{MSMD} we use the choice of $D_x= WENO^{+}$  defined in \eqref{WENO+}, followed by $\tilde{D}_x=WENO^{-}$  
defined in \eqref{WENO-},
by analogy to the first order finite difference  for the linear advection case $U_t = U_x$, 
where we use a differentiation operator $D_x^{+}$  followed by the downwind differentiation operator $D_x^{-}$ to produce a centered difference
for the second derivative. In fact, this approach makes sense for these cases because it respects the properties of the flux for the second derivative 
and consequently satisfies the second derivative condition \eqref{SD}. 
However, if we simply wish the Taylor series formulation to satisfy a 
TVD-like condition,  we are free to use the same operator ($WENO^{+}$  or $WENO^{-}$, as appropriate)  twice, 
and indeed this gives a larger allowable $\dt$.

In Figure \ref{PPvsPM} we show how using the repeated upwind discretization  $D= WENO^{-}$ and $\tilde{D}_x = WENO^{-}$ (solid lines) 
which satisfy the Taylor Series Condition \eqref{TS}  but not the second derivative  condition \eqref{SD} to approximate the higher order 
derivative  allows for a larger time-step than the spatial discretizations (dashed lines) used in \ref{MSMD}. 
We see that for the fifth order methods the rise in total variation always occurs  for larger $\lambda$   for the solid lines  
($\tilde{D}_x=D_x=WENO^{-}$) than  for the dashed lines  ($D_x=WENO^{-}$ and $\tilde{D}_x= WENO^{+}$), even for the method 
designed in \cite{MSMD} to be SSP for the second case but not the first case.
For the sixth order methods the results are almost the same, though the SSP-SD MDRK(4,6,1) method that is SSP for 
base conditions of the type in \cite{MSMD}  performs identically in both cases.
These  results demonstrate that  requiring that the 
spatial discretizations only satisfy \eqref{FE} and  \eqref{TS}  (but not necessarily  \eqref{SD}) results in 
methods with larger allowable time-steps.

\begin{figure}[t!] 
\begin{center} 
\includegraphics[width=0.5\textwidth]{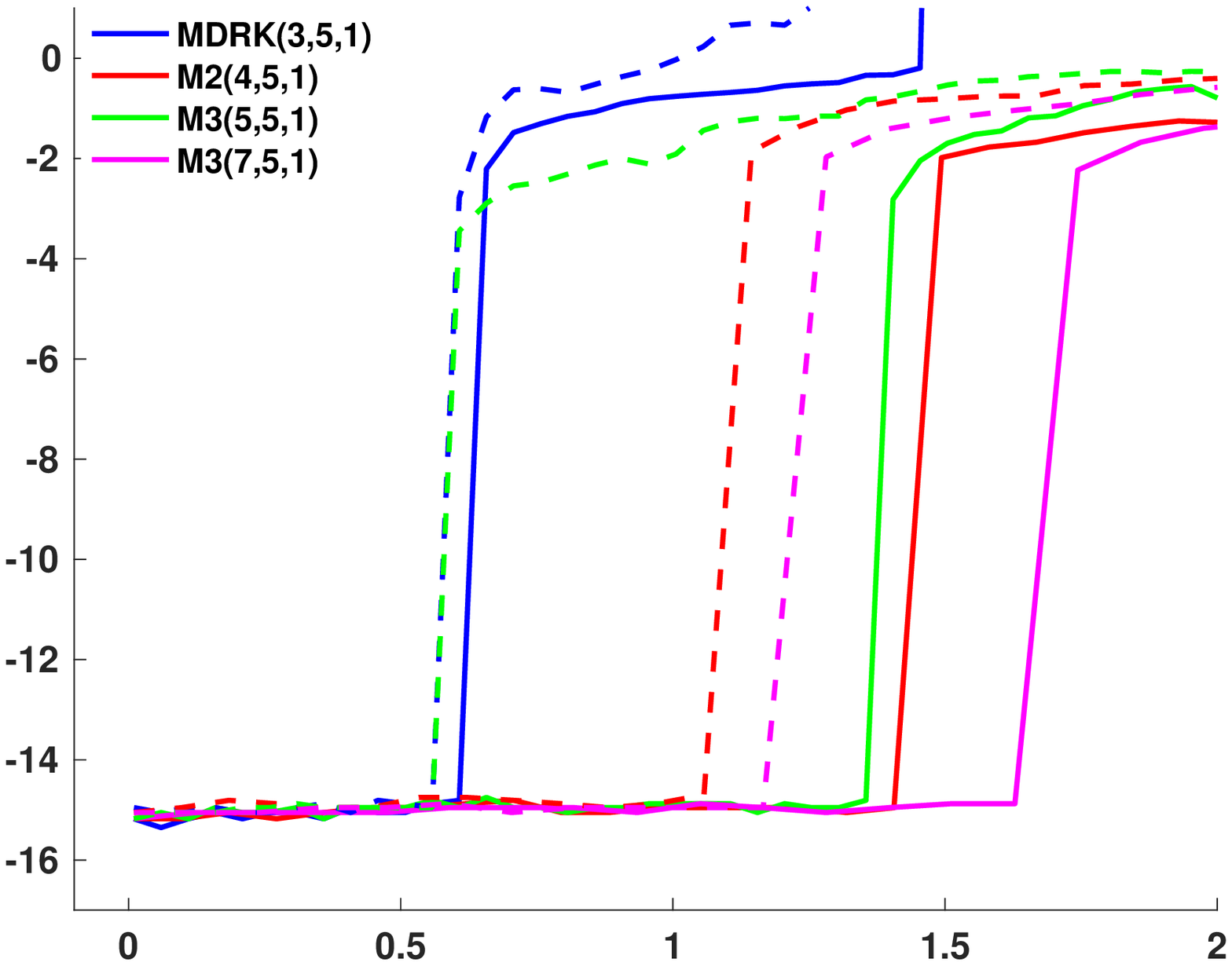}  \hspace{-0.1in}
\includegraphics[width=0.5\textwidth]{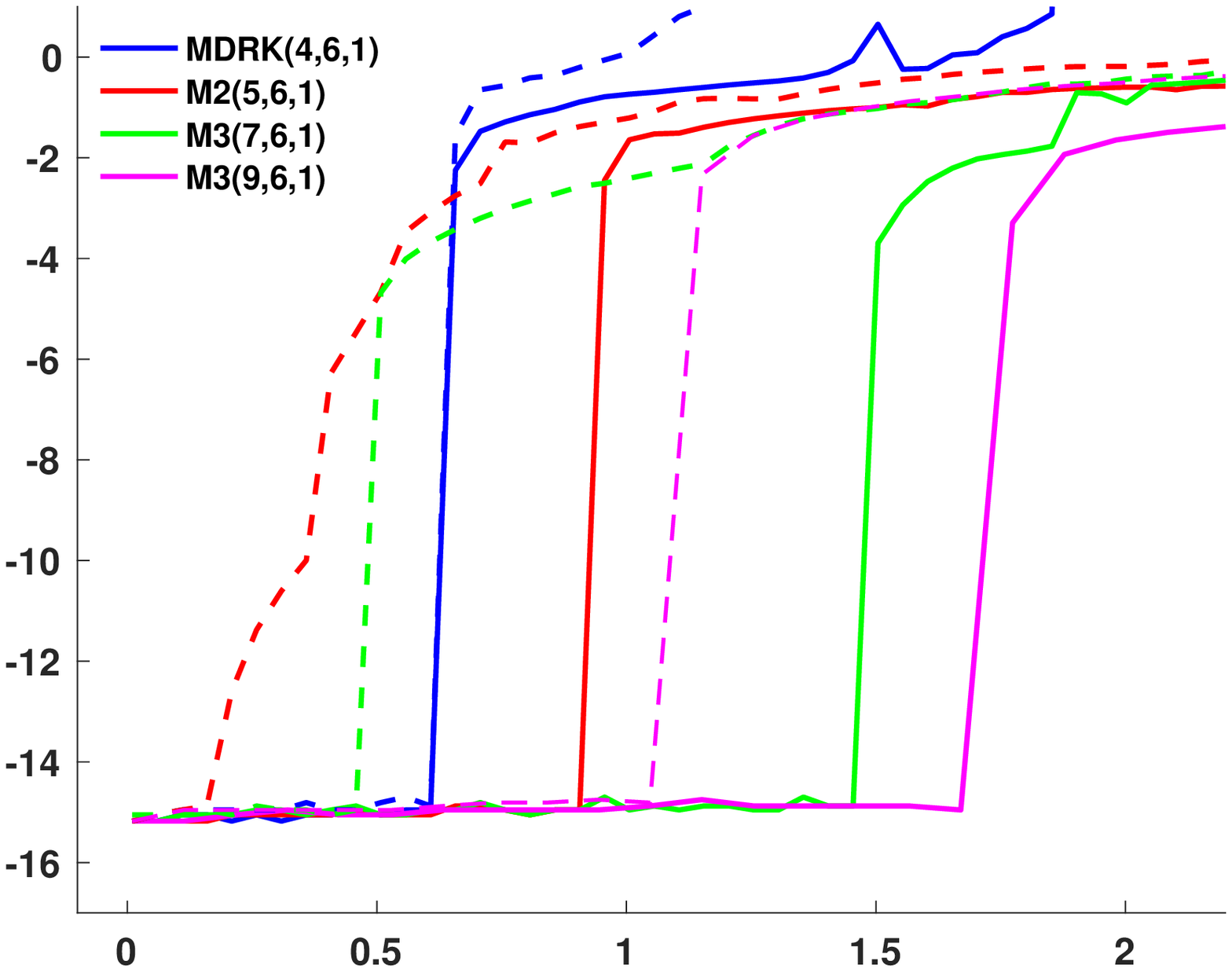} 
\end{center}
\caption{Example 4:  The maximal rise in total variation (on y-axis) for  values of $\lambda$ (on the x-axis).
Simulations using the repeated upwind discretization  $D_x= WENO^{-}$ and $\tilde{D}_x = WENO^{-}$ (solid lines) 
are more efficient than those using $D_x=WENO^{-}$ and $\tilde{D}_x= WENO^{+}$ (dashed lines). This demonstrates the
 enhanced allowable time-step afforded by the SSP-TS methods.}
 \label{PPvsPM}
\end{figure}

\subsection{Example 5: Nonlinear Shallow Water Equations}

As a final test case we consider the shallow water equations, where we are concerned with the preservation of positivity in the numerical solution.  
The shallow water equations \cite{BuKuEtWeDa09} are a non-linear system of hyperbolic conservation laws defined by
\begin{align*}
\begin{pmatrix}
h \\ hv
\end{pmatrix}_t +
\begin{pmatrix}
hv \\ hv^2 + \frac{1}{2} g h^2
\end{pmatrix}_x
=
\begin{pmatrix} 0 \\ 0 \end{pmatrix},
\end{align*}
where $h(x,t)$ denotes the water height at location $x$ and time $t$, 
$v(x,t)$ the water velocity, $g$ is the gravitational constant, and $U = (h, h v)^T$ is the vector of unknown conserved variables.
In our simulations, we set $g=1$.
To discretize this problem, we use the standard Lax-Friedrichs splitting
\[
	\hat{f}_{j- 1/2} := \frac{1}{2} \left( f( u_j ) + f( u_{j-1} ) \right) - \frac{\alpha}{2} \left( u_j - u_{j-1} \right),
	\quad
	\alpha = \max_{j} \left\{ |v_j \pm \sqrt{h_j}| \right\},
\]
and define the (conservative) approximation to the first derivative as
\[
	f(U( x_j ) )_x \approx \frac{1}{\Delta x} \left( \hat{f}_{j+1/2} - \hat{f}_{j-1/2} \right).
\]
We discretize the spatial grid  $x\in (0,1)$ with $M=201$ points.
To approximate the second derivative, we start with element-wise first derivative 
$u_{j,t} := -\frac{1}{ {\Delta x} } \left( \hat{f}_{j+1/2} - \hat{f}_{j-1/2} \right)$, and then 
approximate the second derivative (consistent with Eqn. \eqref{eqn:Ft}) as
\[
	u_{j, tt} := -\frac{1}{2 \Delta x} \left( f'(u_{j+1}) u_{j+1,t}  - f'( u_{j-1} ) u_{j-1,t} \right),
\]
where $f'(u_{j\pm1})$ is the Jacobian of the flux function evaluated at $u_{j\pm 1}$.
A simple first order spatial discretization is chosen here because it enables us to show that positivity is preserved for 
forward Euler and Taylor series for $\lambda^+_{FE} =  \lambda^+_{TS}= \alpha \frac{\Delta t}{\Delta x} \leq 1$.

In problems such as the shallow water equations, the non-negativity of the numerical solution is important 
as a height of $h<0$ is not physically meaningful, and the system loses hyperbolicity when the height becomes negative.
For a positivity preserving test case, we consider a Riemann problem with zero initial velocity, but
with a wet and a dry state \cite{BuKuEtWeDa09,XingZhangShu2010}:
\[
\left( h, v \right)^T = \begin{cases}
    (10,0)^T & x \leq 0.5, \\
    (0,0)^T & x > 0.5.
\end{cases}
\]
In our numerical simulations, we focus on  the impact  of the numerical scheme
on the positivity of the solver for the the water height $h(x,t)$. This quantity is of interest from a numerical perspective because 
if the height $h(x,t) < 0$ for any $x$ or $t$,  the code will crash due to square-root of height.

\begin{table}[t]
\begin{center}
\begin{tabular}{|l|ll||l|ll|} \hline 
{\bf Method} & $\lambda^{pred}$ & $\lambda^{obs}$ & {\bf Method}  & $\lambda^{pred}$ &$\lambda^{obs} $  \\ \hline 
Forward Euler & 1.00000 & 1.01058 & Taylor series & 1.00000 & 1.02598 \\ \hline 
Dormand Prince & 0.00000 & 0.00000 & nonSSPRK(8,6) & 0.00000 & 0.00000 \\ \hline 
SSP-SD MDRK(3,5,2) & --  & 1.03176  & SSP-SD MDRK(4,6,1) & --  & 1.07803 \\ \hline 
SSP-TS M2(4,5,1) & 2.18648 &  3.01005 & SSP-TS M2(5,6,1) & 0.35001 & 2.48411  \\ \hline 
SSP-TS M3(5,5,1) & 1.06253  &  1.78593 & SSP-TS M3(7,6,1) & 0.89468 &  1.64084 \\ \hline 
SSP-TS M3(6,5,1) &  1.82079  &  2.12579 & SSP-TS M3(9,6,1) & 2.59860 & 3.03387   \\ \hline  
\end{tabular}
\caption{The predicted and observed values of $ \lambda = \alpha \frac{\dt}{\dx} $ (where $\dx=\frac{1}{200}$)
for which positivity of the height of the water is preserved in 
the shallow water equations in Example 5. \label{tab:sw-positivity} }
\end{center}
\end{table}

First, we  investigate the behavior of the base methods in terms of the positivity preserving time step.
In other words, we want to get a numerical value for $\DtFE, K$. To do so, we numerically study the positivity behavior of
the  forward Euler and Taylor series approach. To do this, we evolve the solution forward for more time-steps with different values
of  $ \lambda = \alpha \frac{\dt}{\dx} $ to identify the predicted positivity preserving value
$ \lambda^{pred}$. Using the approach, we see that as we increase the number of steps 
the predicted value of the positivity preserving value, $ \lambda^{pred}_{FE} \rightarrow 1$
and $ \lambda^{pred}_{TS} \rightarrow 1$, for both forward Euler and Taylor series.
We are not able to numerically identify $\tilde{K}$  resulting from the 
second derivative condition,  which cannot be evolved forward as it does not approximate the solution to the ODE at all.

In Table \ref{tab:sw-positivity}  we  compare the positivity preserving time-step of a variety of numerical time integrators.
We consider the fifth order SSP-TS methods M2(4,5,1), M3(5,5,1), and M3(6,5,1), and compare their performance to the 
SSP-SD MDRK(3,5,2) method in \cite{MSMD},
and the non-SSP Dormand Prince method. We also consider the sixth order SSP-TS methods M2(5,6,1), M3(7,6,1), and M3(9,6,1), as well as
the SSP-SD MDRK(4,6,1) from \cite{MSMD} and the non-SSPRK(8,6) method.
Positivity of the water height is measure at each stage for a total of $N=60$ time-steps. 
We report the largest allowable value of   $\lambda = \alpha \frac{\dt}{\dx} $ ($\alpha$ is the maximal wavespeed for the domain)  
for which the  solution remains positive.  
For each method, the predicted values $\lambda^{pred}$ are obtained by  multiplying the SSP coefficient $\sspcoef_{TS}$ of that method
by  $\lambda^{pred}_{FE}  =  \lambda^{pred}_{TS}  = 1$. For the SSP-SD MDRK methods we do not make a prediction as
we are not able to identify the $\tilde{K}$  resulting from the  second derivative condition.

 In Table \ref{tab:sw-positivity} we show that 
all of our SSP-TS methods preserve the positivity of the solution for values larger than those predicted by 
the theory $\lambda^{obs} > \lambda^{pred}$,  and that even for the SSP MSRK methods there is a large 
region of values  $\lambda^{obs} $ for which the  solution remains positive.  
However, the non-SSP methods permit no positive time step that retains positivity of the solution, 
highlighting the importance of SSP methods.

\section{Conclusions}

In \cite{MSMD} we introduced a formulation and base conditions to extend the SSP framework to multistage multi-derivative time-stepping methods,
and the resulting SSP-SD methods. 
While the choice of base conditions we used in  \cite{MSMD} gives us more flexibility in finding  SSP time stepping schemes, 
it limits the flexibility in the choice of the spatial discretization.  
In the current paper we introduce an alternative SSP formulation based on the  conditions \eqref{FE} and \eqref{TS} and 
 investigate the resulting explicit two-derivative multistage SSP-TS  time integrators. 
 These base conditions are relevant because some commonly used spatial discretizations may not satisfy the second derivative 
 condition \eqref{SD} which we required in \cite{MSMD}, but do satisfy the Taylor series condition \eqref{TS}. 
 This approach decreases the flexibility in our choice of time discretization because some
time discretizations that can be decomposed into convex combinations of \eqref{FE} and  \eqref{SD}
cannot be decomposed into convex combinations of \eqref{FE} and  \eqref{TS}. However, it increases the
 flexibility in our choice of spatial  discretizations, as we may now consider spatial methods that 
satisfy \eqref{FE} and  \eqref{TS} but not \eqref{SD}.
In the  numerical tests we showed that this increased flexibility allowed for more efficient simulations in several cases.
 
 In this paper, we proved that explicit SSP-TS  methods have a maximum obtainable order of $p=6$.
Next we formulated the proper optimization procedure to generate  SSP-TS  methods.  
Within this new class we were able to organize our schemes into three sub categories
that reflect the different simplifications used in the optimization. We obtained methods up to and including order 
$p=6$ thus breaking the SSP order barrier for explicit SSP Runge-Kutta methods. 
Our  numerical tests  show that the SSP-TS  explicit two-derivative methods perform as expected, preserving the strong stability 
properties satisfied by the base  conditions \eqref{FE} and \eqref{TS} under the predicted 
time-step conditions. Our simulations demonstrate the sharpness of the 
SSP-TS condition in some cases, and  the need for SSP-TS time-stepping methods.  
Furthermore the numerical results indicate that the added freedom in the choice of spatial discretization 
results in larger allowable time steps.
The coefficients of the SSP-TS methods described in this work can be downloaded from \cite{SSPTSgithub}.

\bigskip

{\bf Acknowledgements.}  The work of D.C. Seal was supported in part by the Naval Academy Research Council.
The work of S. Gottlieb and Z.J. Grant was  supported by the AFOSR grant \#FA9550-15-1-0235.

A part of this research is sponsored by the Office of Advanced Scientific Computing Research; U.S. Department of Energy, 
and was performed at the Oak Ridge National Laboratory, which is managed by UT-Battelle, LLC under Contract No. De-AC05-00OR22725.

 Notice: This manuscript has been authored by UT-Battelle, LLC, under contract DE-AC05-00OR22725 with the U.S. Department of Energy. The United States Government retains and the publisher, by accepting the article for publication, acknowledges that the United States Government retains a non-exclusive, paid-up, irrevocable, world-wide license to publish or reproduce the published form of this manuscript, or allow others to do so, for United States Government purposes.
 
\newpage
\appendix
\section{Order Conditions} \label{MDRKOrderCondition}
Any method of the form \eqref{MSMD} must satisfy the order conditions for all $p \leq P$ to be of order $P$.

\bigskip
\noindent $p = 1$ 
\begin{align*}
 b^T e =1  \hspace*{5.15in} 
 \end{align*}
 
 \noindent $p = 2$ 
\begin{align*}
 b^T c+\hat{b}^Te =\frac{1}{2}   \hspace*{4.65in} 
 \end{align*}

 \noindent $p = 3$ 
\begin{subequations}
\begin{align*}
 & \hspace*{-2.25in}  b^T c^2 + 2\hat{b}^T c=\frac{1}{3} \hspace*{2.5in}  \\ 
 & \hspace*{-2.25in}  b^TAc+b^T\hat{c}+\hat{b}^Tc=\frac{1}{6}  
 \end{align*}
  \end{subequations}
  
   \noindent $p = 4$ 
  \begin{subequations}
 \begin{align*}
&  b^Tc^3+3\hat{b}^Tc^2=\frac{1}{4}   \hspace*{4.5in}  \\
&  b^T\left( c \odot Ac\right)+b^T\left( c \odot \hat{c}\right)+\hat{b}^Tc^2+\hat{b}^TAc+\hat{b}^T\hat{c}=\frac{1}{8} \\
	 & b^TAc^2+2b^T\hat{A}c+\hat{b}^Tc^2=\frac{1}{12} \\
&   b^TA^2c+ b^TA\hat{c}+ b^T\hat{A}c+\hat{b}^TAc+\hat{b}^T\hat{c}=\frac{1}{24}
 \end{align*}
 \end{subequations}
 
    \noindent $p = 5$ 
  \begin{subequations}
 \begin{align*}
	&       b^Tc^4 + 4\hat{b}^Tc^3 =\frac{1}{5} \label{p5.1} \\
	           &  b^T\left( c^2 \odot Ac\right)  + b^T\left( c^2 \odot \hat{c}\right)+\hat{b}^Tc^3+2\hat{b}^T\left( c \odot Ac\right)+2\hat{b}^T\left( c \odot \hat{c}\right)=\frac{1}{10}   \\
	           &  b^T\left ( c \odot Ac^2\right)+2b^T\left(c \odot \hat{A}c\right)+\hat{b}^Tc^3+\hat{b}^T Ac^2+2\hat{b}^T\hat{A}c=\frac{1}{15} \\
	           &  b^T\left( c \odot A^2c \right)+b^T\left( c \odot A\hat{c}\right)+b^T\left( c \odot \hat{A}c\right)+\hat{b}^T\left( c \odot Ac\right)+\hat{b}^T\left(c \odot \hat{c}\right) 
	         + \hat{b}^TA^2c+\hat{b}^TA\hat{c}+\hat{b}^T\hat{A}c=\frac{1}{30} \\
	           &  b^T\left( Ac \odot Ac \right) +2b^T\left( \hat{c} \odot Ac\right)+b^T\hat{c}^2+ 2\hat{b}^T\left(c \odot Ac\right)+2\hat{b}^T\left(c \odot \hat{c}\right)=\frac{1}{20}   \\
	          &  b^TAc^3+3b^T\hat{A}c^2+\hat{b}^Tc^3=\frac{1}{20}  \\
	           &  b^TA\left(c	 \odot  Ac\right)+b^TA\left(c \odot \hat{c}\right)+b^T\hat{A}c^2+b^T\hat{A}Ac+b^T\hat{A}\hat{c}  + \hat{b}^T\left( c \odot Ac\right)+\hat{b}^T\left( c \odot \hat{c}\right)=\frac{1}{40} \\
	& b^TA^2c^2+2b^TA\hat{A}c+b^T\hat{A}c^2+\hat{b}^T Ac^2+2\hat{b}^T\hat{A}c=\frac{1}{60} \\
&	 b^TA^3 c+b^TA^2 \hat{c}+b^T A \hat{A}c+b^T\hat{A}Ac+b^T\hat{A}\hat{c}
	+\hat{b}^TA^2c +\hat{b}^TA\hat{c}+\hat{b}^T\hat{A}c=\frac{1}{120} 
 \end{align*}
 \end{subequations}
 
     \noindent $p = 6$  
 \begin{subequations}
 \begin{align*} 
	        	  &      b^T c^5 +  5\hat{b}^T c^4 =  \frac{1}{6}   \\
	           &  b^T \left(c^3 \odot Ac\right) + 3\hat{b}^T \left(c^2 \odot Ac \right) + \hat{b}^T c^4 + b^T \left(c^3 \odot \hat{c}\right) + 3\hat{b}^T \left(c^2 \odot \hat{c}\right)  =  \frac{1}{12}\\
	           &  b^T \left( c^2 \odot A c^2\right) + 2\hat{b}^T\left(c \odot A c^2\right) + 2b^T \left(c^2\odot \hat{A}c\right) +  \hat{b}^T c^4 + 4\hat{b}^T\left(c \odot \hat{A}c\right)  =  \frac{1}{18} \\
	           &  b^T\left( c \odot A c^3\right) + 3b^T\left( c \odot \hat{A} c^2\right) + \hat{b}^TA c^3 + 3\hat{b}^T\hat{A} c^2 + \hat{b}^T c^4  =  \frac{1}{24} \\
	          &  b^TA c^4 +  4b^T\hat{A} c^3 +  \hat{b}^T c^4  =  \frac{1}{30}\\
	           &  b^T\left( c^2 \odot A^2c\right) + 2\hat{b}^T\left(c \odot A^2c\right) + b^T \left(c^2 \odot A\hat{c}\right) + b^T \left(c^2 \odot \hat{A}c\right) + \hat{b}^T \left(c^2 \odot Ac\right)  + 2\hat{b}^T\left(c  \odot A\hat{c}\right)   \nonumber \\
		  &	\; \; \; \; \; \; \; \; \; \; \; \; \; \;  + 2\hat{b}^T\left(c \odot \hat{A}c\right) + \hat{b}^T \left( c^2 \odot \hat{c}\right)  =  \frac{1}{36} \\	
		  & b^T\left(c \odot A^2 c^2\right) + \hat{b}^TA^2 c^2 + \hat{b}^T\left( c \odot A c^2\right) + b^T\left( c \odot \hat{A} c^2\right) + 2b^T\left( c \odot A\hat{A}c\right) + \hat{b}^T\hat{A} c^2 + 2\hat{b}^TA\hat{A}c \nonumber \\
		  & \; \; \; \; \; \; \; \; \; \; \; \; \; \;   + 2\hat{b}^T\left(c \odot \hat{A}c\right) =\frac{1}{72}\\
		  & b^TA^2 c^3 +  \hat{b}^TA c^3 +  b^T\hat{A} c^3 +  3b^TA\hat{A} c^2 +  3\hat{b}^T\hat{A} c^2  =  \frac{1}{120}\\
		  & b^T\left( c \odot Ac \odot Ac\right) + \hat{b}^T\left(Ac \odot Ac\right)  + b^T\left(c \odot \hat{A}Ac\right) + b^T\left( c\odot A\left(c\odot\hat{c}\right)\right) + \hat{b}^T \left(c^2 \odot Ac\right) + b^T\left( c \odot \hat{A} c^2\right) 				\nonumber \\	
		  & \; \; \; \; \; \; \; \; \; \; \; \; \; \; +\hat{b}^T\hat{A}Ac + \hat{b}^TA\left(c \odot \hat{c}\right) + b^T\left( c \odot \hat{A}\hat{c}\right) + \hat{b}^T\hat{A} c^2 + \hat{b}^T c^2\hat{c} +  \hat{b}^T\hat{A}\hat{c}  =  \frac{1}{48}    \\		  	  
	          &b^TA \left(c^2 \odot Ac\right) + b^TA \left(c^2 \odot \hat{c}\right) + \hat{b}^T \left(c^2 \odot Ac\right) + 2b^T\left(\hat{A}c \odot Ac\right) + b^T\hat{A} c^3 + 2b^T\left(\hat{A}c \odot \hat{c}\right) \nonumber \\
	          &	\; \; \; \; \; \; \; \; \; \; \; \; \; \; + \hat{b}^T \left( c^2 \odot \hat{c}\right)  =  \frac{1}{60}\\
	          & b^T\left(Ac \odot A c^2\right) +  \hat{b}^T\left( c \odot A c^2\right) + b^T\hat{A}A c^2 + b^T\hat{A} c^3 + 2b^T\left(Ac \odot \hat{A}c\right) + 2b^T\hat{A}^2c + 2\hat{b}^T\left(c \odot \hat{A}c\right)  =  \frac{1}{90}      \\	          
	          & b^T\left(c \odot A^3c\right) + \hat{b}^TA^3c + \hat{b}^T\left( c \odot A^2c \right)+ b^T\left(  c \odot \hat{A}Ac\right) + b^T\left( c \odot A\hat{A}c\right) + b^T\left(c \odot A^2\hat{c}\right) + \hat{b}^T\hat{A}Ac \nonumber  \\ 
	          &\; \; \; \; \; \; \; \; \; \; \; \; \; \; + \hat{b}^TA\hat{A}c  + \hat{b}^TA^2\hat{c} + b^T\left(c \odot \hat{A}\hat{c}\right) + \hat{b}^TCA\hat{c} + \hat{b}^TC\hat{A}c + \hat{b}^T\hat{A}\hat{c}  =  \frac{1}{144}\\	         	          	          	          
	          & b^T\left( Ac \odot A^2c\right) + b^T\left( Ac \odot A\hat{c}\right) + b^T\left(Ac \odot \hat{A}c\right) + b^T\left(\hat{A}c \odot Ac\right) + b^T\hat{A}A^2c + \hat{b}^T\left(c \odot A^2c\right) \nonumber \\
	          &\; \; \; \; \; \; \; \; \; \; \; \; \; \;  + b^T\left( \hat{A}c \odot \hat{c}\right)  + b^T\hat{A}A\hat{c} + \hat{b}^T\left(c \odot A\hat{c}\right) + b^T\hat{A}^2c + \hat{b}^T\left(c \odot \hat{A}c\right)  =  \frac{1}{180}\\
	          & b^T(A^2\left(c \odot Ac\right) + b^TA^2\left(c \odot \hat{c}\right) + b^TA\hat{A}c^2 + b^TA\hat{A}Ac + b^T\hat{A}\left(c \odot Ac\right) + \hat{b}^TA\left(c \odot Ac\right) + b^TA\hat{A}\hat{c} \nonumber \\
	       	 & \; \; \; \; \; \; \; \; \; \; \; \; \; \;   + b^T\hat{A}\left(c \odot \hat{c}\right)  + \hat{b}^TA\left(c \odot \hat{c}\right)  + \hat{b}^T\hat{A}c^2  + \hat{b}^T\hat{A}Ac + \hat{b}^T\hat{A}\hat{c}  =  \frac{1}{240}\\
	          & b^TA^3 c^2 +  \hat{b}^TA^2 c^2 +  b^T\hat{A}A c^2 +  b^TA\hat{A} c^2 +  2b^TA^2\hat{A}c +  \hat{b}^T\hat{A} c^2 +  2\hat{b}^TA\hat{A}c +  2b^T\hat{A}^2c  =  \frac{1}{360}           \\
	          & b^T\left(c \odot  Ac \odot  Ac\right) + \hat{b}^T\left(Ac \odot  Ac\right) + 2b^T\left(c \odot  \hat{c} \odot  Ac\right) + 2\hat{b}^T\left(c^2 \odot  Ac\right) + 2\hat{b}^T\left(\hat{c} \odot  Ac\right) + 2\hat{b}^T\left(c^2 \odot  \hat{c}\				\right)\nonumber \\
	          & \; \; \; \; \; \; \; \; \; \; \; \; \; \;  + b^T\left(c \odot  \hat{c}^2\right) + \hat{b}^T \hat{c}^2  =  \frac{1}{24}                   
	       \end{align*}
 \end{subequations}
  \begin{subequations}
 \begin{align*} 	
  	 & b^T\left(Ac \odot  A^2c\right) + b^T\left(\hat{c} \odot  A^2c\right) + \hat{b}^T\left(c \odot A^2c\right) + \hat{b}^T\left(Ac \odot  Ac\right) + b^T\left(Ac \odot  \hat{A}c\right) + b^T\left(Ac \odot  A\hat{c}\right) \nonumber \\
	         & \; \; \; \; \; \; \; \; \; \; \; \; \; \; + 2 \hat{b}^T\left(\hat{c} \odot  Ac\right) + b^T\left(\hat{c} \odot  \hat{A}c\right) + b^T\left(\hat{c} \odot  A\hat{c}\right) + \hat{b}^T\left(c \odot  \hat{A}c\right) + \hat{b}^T\left(c \odot  A\hat{c}\right) 
	         + \hat{b}^T \hat{c}^2  =  \frac{1}{72}\\            
	       	& b^T\left(Ac \odot  Ac^2\right) + b^T\left(\hat{c} \odot  Ac^2\right) + \hat{b}^T\left(c \odot  Ac^2\right) + \hat{b}^T\left(Ac \odot  c^2\right) + 2b^T\left(Ac \odot  \hat{A}c\right) + \hat{b}^T\left(\hat{c} \odot  c^2\right) \nonumber \\
	        & \; \; \; \; \; \; \; \; \; \; \; \; \; \;  + 2b^T\left(\hat{c} \odot  \hat{A}c\right) + 2\hat{b}^T\left(c \odot  \hat{A}c\right)  =  \frac{1}{36}\\
   	       	& b^TA\left(Ac \odot  Ac\right) + 2b^TA\left(\hat{c} \odot  Ac\right) + 2b^T\hat{A}\left(c \odot  Ac\right) + \hat{b}^T\left(Ac \odot  Ac\right) + 2\hat{b}^T\left(\hat{c} \odot  Ac\right) + 2b^T\hat{A}\left(c \odot  \hat{c}\right) \nonumber \\
	        & \; \; \; \; \; \; \; \; \; \; \; \; \; \;  + b^TA\hat{c}^2 + \hat{b}^T \hat{c}^2  =  \frac{1}{120}\\
	      	& b^TA^4c +  b^TA^3\hat{c} +  b^TA^2\hat{A}c +  b^TA\hat{A}Ac +  b^T\hat{A}A^2c +  \hat{b}^TA^3c +  b^TA\hat{A}\hat{c} +  b^T\hat{A}A\hat{c} +  \hat{b}^TA^2\hat{c} \nonumber \\
	       & \; \; \; \; \; \; \; \; \; \; \; \; \; \; +  b^T\hat{A}^2c +  \hat{b}^TA\hat{A}c  +  \hat{b}^T\hat{A}Ac +  \hat{b}^T\hat{A}\hat{c}  =  \frac{1}{720}  
	       \end{align*}
 \end{subequations} 

\newpage
\section{Coefficients of selected methods}
All the time-stepping methods in this work can be downloaded as Matlab files from \cite{SSPTSgithub}. In this appendix we present selected methods.

\bigskip
\noindent {\bf SSP-TS M2(4,5,1) :} This method has  $\sspcoef_{TS}= 2.18648$

\[
\begin{array}{lll}
a_{21}=4.280141748183123e-01 \; \; &  \hat{a}_{21}=9.159806692270039e-02  \; \; & b_{1}=3.456442194983256e-01  \vspace{.1cm}  \\ 
a_{31}=3.174364422211321e-01 & \hat{a}_{31}=2.068159838961376e-02 &  b_{2}=1.551487425849178e-01  \vspace{.1cm}  \\
a_{32}=1.032647478325804e-01   &  \hat{a}_{32}=2.361437143530821e-02 &    b_{3}=3.458932447335502e-01  \vspace{.1cm} \\
a_{41}=3.280547501426051e-01 & \hat{a}_{41}=1.869435227642530e-02 &  b_{4}=1.533137931832064e-01   \vspace{.1cm} \\
a_{42}=9.334228125655676e-02 & \hat{a}_{42}=2.134532206271365e-02  &  \hat{b}_{1}=3.226836941745746e-02  \vspace{.1cm}\\
a_{43}=4.134096583922347e-01 & \hat{a}_{43}=9.453767556809974e-02  &\hat{b}_2=1.785928934720153e-02   \vspace{.1cm}\\
& &  \hat{b}_3=7.490191551289183e-02  \vspace{.1cm} \\
& &  \hat{b}_4=3.505948481328697e-02  \vspace{.1cm} \\

\end{array}
\] 

\bigskip

\noindent {\bf SSP-TS M3(8,6,1):} This method has   $\sspcoef_{TS}= 1.7369$ 
\[
\begin{array}{lll}
a_{21}=3.498630949258150e-01\;\; &  a_{65}=3.779563241192044e-01 \;\;&  \hat{a}_{21}=6.120209259553491e-02 \vspace{.1cm}\\
a_{31}=2.253295269463227e-01 & a_{71}=2.148681581922796e-01 &  \hat{a}_{31}=1.921063160949869e-02 \vspace{.1cm}\\
a_{32}=1.807161013759724e-01 & a_{72}=1.533420472452636e-01 &\hat{a}_{41}=4.358605297856505e-03 \vspace{.1cm}\\
a_{41}=2.071695605568409e-01 & a_{73}=1.813808417863181e-02 & \hat{a}_{51}=2.136333816692593e-03 \vspace{.1cm}\\
a_{42}=4.100178308548576e-02 & a_{74}=7.994387176143736e-02 & \hat{a}_{61}=1.402456855983780e-03 \vspace{.1cm}\\
a_{43}=1.306253212278126e-01 & a_{75}=1.630752796649391e-01 &\hat{a}_{71}=1.631142330728269e-02 \vspace{.1cm}\\
a_{51}=1.667117585911237e-01 & a_{76}=2.484093806816690e-01 &\hat{a}_{81}=1.548804492637956e-02 \vspace{.1cm}\\
a_{52}=2.009667996165993e-02 & a_{81}=2.036762412289922e-01 &b_1\;\;=1.927179349665056e-01 \vspace{.1cm}\\ 
a_{53}=6.402490521280881e-02 & a_{82}=1.456707401767411e-01 &b_2\;\;=7.457643792836192e-02 \vspace{.1cm}\\
a_{54}=2.821909187189924e-01 & a_{83}=2.379744031395224e-02 &b_3 \;\;=1.097549250079706e-01 \vspace{.1cm}\\
a_{61}=1.493141923275556e-01 & a_{84}=1.048777345557326e-01 &b_4\;\;=1.166274027628658e-01 \vspace{.1cm}\\
a_{62}=1.319303489675465e-02 & a_{85}=2.139668745571685e-01 &b_5 \;\;=1.862061970475841e-01 \vspace{.1cm}\\
a_{63}=4.203095914776495e-02 & a_{86}=6.560681670556633e-02 & b_6\;\;=1.088089628270683e-01 \vspace{.1cm}\\
a_{64}=1.852522020371737e-01 & a_{87}=1.520556075200664e-01 & b_7\;\;=4.414821350738243e-02 \vspace{.1cm}\\
&&b_8\;\;=1.671599259522612e-01 \vspace{.1cm}\\
&&\hat{b}_1\;\;=1.156518516980132e-02\vspace{.1cm}\\
\end{array}
\]

\newpage
\section{Fifth order WENO method of Jiang and Shu} \label{WENO}

To solve the PDE $$ u_t + f(u)_x = 0$$ we approximate the spatial derivative to obtain 
$F(u) \approx -f(u)_x$, and then use
a time-stepping method to solve the resulting system of ODEs.  In this section we describe the fifth order WENO
scheme presented in \cite{jiang1996}.

First, we split the flux into the positive and negative parts
$$f(u) = f^{+}(u) + f^{-}(u).$$
This can be accomplished in a variety of ways, e.g. the Lax-Friedrichs
flux splitting 
$$f^{+}(u) = \frac{1}{2} \left( f(u) + m u \right) \; \; \; \;
f^{-}(u) = \frac{1}{2} \left( f(u) - m u \right),$$
where $m = \max |f'(u)|$.  In this way we ensure that 
$\frac{d f^{+}}{du} \geq 0$ and $\frac{d f^{-}}{du} \leq 0$.

To calculate the  numerical fluxes $\hat{f}^{+}_{j+ \frac{1}{2}}$ and $\hat{f}^{-}_{j+ \frac{1}{2}}$,
we begin by calculating the smoothness measurements to determine if a shock lies
within the stencil.  For our fifth order scheme, these are:
\begin{eqnarray*}
IS_0^+ & = & \frac{13}{12} \left( f^+_{j-2} -2 f^+_{j-1} + f^+_{j} \right)^2
          + \frac{1}{4} \left( f^+_{j-2} -4 f^+_{j-1} + 3 f^+_{j} \right)^2 \\
IS_1^+ & = & \frac{13}{12} \left( f^+_{j-1} -2 f^+_{j} + f^+_{j+1} \right)^2
          + \frac{1}{4} \left( f^+_{j-1} - f^+_{j+1} \right)^2 \\
IS_2^+ & = & \frac{13}{12} \left( f^+_{j} -2 f^+_{j+1} + f^+_{j+2} \right)^2
          + \frac{1}{4} \left( 3 f^+_{j} -4 f^+_{j+1} +  f^+_{j+2} \right)^2 
\end{eqnarray*}
and
\begin{eqnarray*}
IS_0^- & = & \frac{13}{12} \left( f^-_{j+1} -2 f^-_{j+2} + f^-_{j+3} \right)^2
          + \frac{1}{4} \left( 3 f^-_{j+1} -4 f^-_{j+2} +  f^-_{j+3} \right)^2\\
IS_1^- & = & \frac{13}{12} \left( f^-_{j} -2 f^-_{j+1} + f^-_{j+2} \right)^2
          + \frac{1}{4} \left( f^-_{j} - f^-_{j+2} \right)^2 \\
IS_2^- & = & \frac{13}{12} \left( f^-_{j-1} -2 f^-_{j} + f^-_{j+1} \right)^2
          + \frac{1}{4} \left( f^-_{j-1} -4 f^-_{j} + 3 f^-_{j+1} \right)^2 \; .
\end{eqnarray*}
Next, we use the smoothness measurements to calculate the stencil weights
$$ \alpha^\pm_0 = \frac{1}{10} \left( \frac{1}{\epsilon + IS_0^\pm}\right)^2 \; \; \; \; \;
\alpha^\pm_1 = \frac{6}{10} \left( \frac{1}{\epsilon + IS_1^\pm}\right)^2 \; \; \; \; \;
\alpha^\pm_2 = \frac{3}{10} \left( \frac{1}{\epsilon + IS_2^\pm}\right)^2  $$
and
$$ 
\omega^\pm_0 = \frac{\alpha^\pm_0}{\alpha^\pm_0 + \alpha^\pm_1 + \alpha^\pm_2} \; \; \; \; \;
\omega^\pm_1 = \frac{\alpha^\pm_1}{\alpha^\pm_0 + \alpha^\pm_1 + \alpha^\pm_2} \; \; \; \; \;
\omega^\pm_2 = \frac{\alpha^\pm_2}{\alpha^\pm_0 + \alpha^\pm_1 + \alpha^\pm_2}. $$
Finally, the numerical fluxes are
\begin{eqnarray*}
 \hat{f}^{+}_{j+ \frac{1}{2}} & = &\omega^+_0 
\left(\frac{2}{6} f^+_{j-2} -\frac{7}{6} f^+_{j-1} + \frac{11}{6} f^+_{j} \right)
 +  \omega^+_1 
\left( - \frac{1}{6} f^+_{j-1} + \frac{5}{6} f^+_{j} + \frac{2}{6} f^+_{j+1} \right)\\
&+& \omega^+_2
\left(\frac{2}{6} f^+_{j} + \frac{5}{6} f^+_{j+1} - \frac{1}{6} f^+_{j+2} \right) 
\end{eqnarray*}
and
\begin{eqnarray*}
 \hat{f}^{-}_{j+ \frac{1}{2}} & = &\omega^-_2 
\left( - \frac{1}{6} f^-_{j-1} + \frac{5}{6} f^-_{j} + \frac{2}{6} f^-_{j+1} \right)
 +  \omega^-_1 
\left(\frac{2}{6} f^-_{j} + \frac{5}{6} f^-_{j+1} - \frac{1}{6} f^-_{j+2} \right) \\
&+& \omega^-_0
\left(\frac{11}{6} f^-_{j+1} - \frac{7}{6} f^-_{j+2} + \frac{2}{6} f^-_{j+3} \right) .
\end{eqnarray*}

Finally, we compute
\begin{eqnarray}
WENO^+ f^{+}(u) & = &    -  \frac{1}{\Delta x} \left(
\hat{f}^{+}_{j+ \frac{1}{2}} - \hat{f}^{+}_{j- \frac{1}{2}} \right) 
  \label{WENO+} \\
WENO^- f^{-}(u) & = &  - \frac{1}{\Delta x} \left(
\hat{f}^{-}_{j+ \frac{1}{2}} - \hat{f}^{-}_{j- \frac{1}{2}} \right) 
  \label{WENO-}
\end{eqnarray}
and put it all together
\[ F(u) =  WENO^+ f^{+}(u) + WENO^- f^{-}(u) \approx -f(u)_x .\]


\end{document}